\documentclass{amsart}

\usepackage{amsmath, amssymb, amsthm, marginnote, ltxcmds, adjustbox, stmaryrd, graphicx, bbold, extarrows}
\usepackage{esvect}

\usepackage{mathtools, stackengine, changepage, ragged2e}

\usepackage{todonotes}

\usepackage{mathrsfs}

\definecolor{cite}{HTML}{11871E}
\definecolor{url}{HTML}{698996}
\definecolor{link}{HTML}{912F1B}

\usepackage[pdfencoding=unicode, colorlinks=true, linkcolor=link, citecolor=cite, urlcolor=url, linktocpage]{hyperref}

\usepackage[backend=biber, style=alphabetic, maxnames=10, maxalphanames=3, minalphanames=3]{biblatex}
\usepackage{tikz}
\usetikzlibrary{cd}
\usetikzlibrary{arrows, arrows.meta, positioning, calc}
\tikzcdset{arrow style=tikz, diagrams={>={Straight Barb[scale=0.8]}}}

\tikzstyle{arrow} = [-{Straight Barb[scale=0.8]}, line width=0.2mm]
\tikzset{
	math to/.tip={Glyph[glyph math command=rightarrow]},
	loop/.tip={Glyph[glyph math command=looparrowleft, swap]},
}

\usepackage{enumitem}
	{\end{enumerate}%
}

\usepackage[
letterpaper,
twoside=false,
textheight=22cm,
textwidth=14.4cm,
marginparsep=0.75cm,
marginparwidth=2.5cm,
heightrounded,
centering
]{geometry}

\usepackage{lineno}

\usepackage{contour}
\usepackage[normalem]{ulem}

\contourlength{0.5pt}

\newcommand{\myuline}[1]{%
	\uline{\phantom{#1}}%
	\llap{\contour{white}{#1}}%
}

\makeatletter
\newcommand*{\saved@myuline}{}
\let\saved@myuline\myuline

\newcommand*{\mathuline}{%
	\mathpalette{\math@myuline\saved@myuline}%
}
\newcommand*{\math@myuline}[3]{%
	\mbox{#1{$#2#3\m@th$}}%
}

\renewcommand*{\myuline}{%
	\relax  
	\ifmmode
	\expandafter\mathuline
	\else
	\expandafter\saved@myuline
	\fi
}
\makeatother

\usepackage{libertine}
\usepackage{stmaryrd}

\DeclareFontFamily{T1}{cbgreek}{}
\DeclareFontShape{T1}{cbgreek}{m}{n}{<-6>  grmn0500 <6-7> grmn0600 <7-8> grmn0700 <8-9> grmn0800 <9-10> grmn0900 <10-12> grmn1000 <12-17> grmn1200 <17-> grmn1728}{}
\DeclareSymbolFont{quadratics}{T1}{cbgreek}{m}{n}
\DeclareMathSymbol{\qoppa}{\mathord}{quadratics}{21}

\usepackage[capitalise]{cleveref}

\Crefname{prop}{Proposition}{Propositions}
\Crefname{lem}{Lemma}{Lemmas}
\Crefname{cor}{Corollary}{Corollaries}
\Crefname{thm}{Theorem}{Theorems}
\Crefname{alphThm}{Theorem}{Theorems}

\Crefname{defn}{Definition}{Definitions}
\Crefname{notation}{Notation}{Notations}
\Crefname{cons}{Construction}{Constructions}
\Crefname{rmk}{Remark}{Remarks}
\Crefname{obs}{Observation}{Observations}
\Crefname{trick}{Trick}{Tricks}
\Crefname{warning}{Warning}{Warnings}
\Crefname{conj}{Conjecture}{Conjectures}
\Crefname{assump}{Assumption}{Assumptions}
\Crefname{recollect}{Recollection}{Recollections}
\Crefname{terminology}{Terminology}{Terminologies}

\Crefname{question}{Question}{Questions}
\Crefname{example}{Example}{Examples}

\Crefname{figure}{Figure}{Figures}

\crefformat{equation}{(#2#1#3)}
\crefformat{section}{\S#2#1#3}
\crefmultiformat{section}{\S\S#2#1#3}{and~#2#1#3}{, #2#1#3}{, and~#2#1#3}

\newtheorem{thm}[subsubsection]{Theorem}
\newtheorem{prop}[subsubsection]{Proposition}
\newtheorem{lem}[subsubsection]{Lemma}
\newtheorem{cor}[subsubsection]{Corollary}

\newtheorem{alphThm}{Theorem}

\newcommand{\neutralize}[1]{\expandafter\let\csname c@#1\endcsname\count@}
\makeatother

\newtheorem{introthm}{Theorem}

\newtheorem{introtheorem}[introthm]{Theorem}
\newtheorem{introquestion}[introthm]{Question}

\newtheorem{appendixproposition}{Proposition}[section]
\newtheorem{appendixlemma}{Lemma}[section]

\newtheorem*{thm*}{Theorem}
\newtheorem*{prop*}{Proposition}
\newtheorem*{lem*}{Lemma}
\newtheorem*{cor*}{Corollary}

\newtheorem{alphConj}{Conjecture}



\newtheorem{alphCor}{Corrollary}



\newtheorem{alphProp}{Proposition}



\theoremstyle{definition}
\newtheorem*{defn*}{Definition}
\newtheorem{defn}[subsubsection]{Definition}

\newtheorem{recollect}[subsubsection]{Recollections}

\newtheorem{conj}[subsubsection]{Conjecture}

\theoremstyle{remark}
\newtheorem{rmk}[subsubsection]{Remark}
\newtheorem{obs}[subsubsection]{Observation}
\newtheorem{example}[subsubsection]{Example}

\newtheorem{constr}[subsubsection]{Construction}

\newcommand{\abeliangroups}{\mathrm{Ab}}
\newcommand{\all}{\mathrm{all}}
\newcommand{\Ar}{\mathrm{Ar}}
\newcommand{\asm}{\mathrm{asm}}
\newcommand{\bbR}{\mathbb{R}}

\newcommand{\catp}{\mathrm{Cat}^{\mathrm{p}}}
\newcommand{\codesc}{\mathrm{codesc}}
\newcommand{\cofib}{\mathrm{cofib}}
\newcommand{\connquadL}{\tau_{\geq 1}\mathrm{L}^{\mathrm{q}}}
\newcommand{\coconnquadL}{\tau_{\leq 0}\mathrm{L}^{\mathrm{q}}}
\newcommand{\connquadkarL}{\tau_{\geq 1}\mathbb{L}^{\mathrm{q}}}
\newcommand{\Corr}{\mathrm{Corr}}

\newcommand{\equi}{\mathrm{equiv}}
\newcommand{\equivfiniteness}{w}
\newcommand{\equivprojclass}[2]{\mathrm{Wa}^{#1}(#2)}

\newcommand{\family}{\mathcal{F}}
\newcommand{\fib}{\mathrm{fib}}
\newcommand{\fincov}{\mathrm{fincov}}
\newcommand{\finite}{\mathrm{fin}}
\newcommand{\finitemodule}{\mathrm{f}}
\newcommand{\finpairs}{\mathrm{Fin}^{(2)}}
\newcommand{\finpairstrans}{\mathrm{Fin}^{(2)}_{\mathrm{tf}}}
\newcommand{\fold}{\mathrm{fold}}
\newcommand{\func}{\mathrm{Fun}}
\newcommand{\generichlgy}{\mathcal{H}}
\newcommand{\Gpd}{\mathrm{Gpd}}
\newcommand{\hAut}{\mathrm{hAut}}
\newcommand{\id}{\mathrm{id}}

\newcommand{\induct}{\mathrm{ind}}
\newcommand{\karLthy}{\mathbb{L}}

\newcommand{\Kthy}{\mathrm{k}}
\newcommand{\Lthy}{\mathrm{L}}
\newcommand{\leftclass}{\mathrm{left}}
\newcommand{\map}{\mathrm{map}}
\newcommand{\module}[1]{\mathrm{Mod}_{#1}}
\newcommand{\ncKthy}{\mathrm{K}}
\newcommand{\nctateK}{\mathrm{K}^{tC_2}}
\newcommand{\nielsspace}{X_{\mathrm{Bor}}}
\newcommand{\normfam}{{{\mathcal{F}}_{\mathrm{n}}}}
\newcommand{\maximal}{\mathcal{M}}
\newcommand{\Met}{\mathrm{Met}}
\newcommand{\op}{^{\mathrm{op}}}

\newcommand{\orbit}{\mathrm{Orb}}
\newcommand{\orbifoldfundamentalgroup}{\mathcal{O}}
\newcommand{\Out}{\mathrm{Out}}
\newcommand{\pdbord}{\Omega^{\mathrm{PD}}}
\newcommand{\Poinc}{\mathrm{Pn}}
\newcommand{\PSh}{\mathcal{P}}

\newcommand{\quadratic}{{\qoppa}^{\mathrm{q}}}
\newcommand{\quadL}{{\Lthy}^\mathrm{q}}
\newcommand{\quadkarL}{{\karLthy}^{\mathrm{q}}}
\newcommand{\qoppamet}{{\qoppa}^{\mathrm{met}}}
\newcommand{\reducedK}{\widetilde{\mathrm{K}}}
\newcommand{\restrict}{\mathrm{res}}
\newcommand{\rightclass}{\mathrm{right}}

\newcommand{\spc}{\mathcal{S}}
\newcommand{\spcpairtrans}{\spc_{\mathrm{tf}}^{2}}
\newcommand{\spectra}{\mathrm{Sp}}
\newcommand{\structure}{\mathcal{S}}
\newcommand{\structurekar}{{\structure}_{\mathrm{kar}}}
\newcommand{\structurekarper}{{\structure}_{\mathrm{kar}}^{\mathrm{per}}}
\newcommand{\tateK}{\mathrm{k}^{tC_2}}

\newcommand{\tso}{\mathrm{tso}}
\newcommand{\torsionfree}{{\family_{\mathrm{tf}}}}
\newcommand{\vcyc}{\mathrm{vc}}

\newcommand{\whiteheadgroup}{\mathrm{Wh}}

\def\colim{\qopname\relax m{colim}}

\newcommand{\category}[1]{\mathcal{#1}}

\newcommand{\bbZ}{\mathbb{Z}}

\newcommand{\arrdisp}{0.33ex}
\newcommand{\arrdisplacementsp}{0.72ex}

\newcommand{\ardis}{\ar@<\arrdisp>}
\newcommand{\ardissp}{\ar@<\arrdisplacementsp>}



\title{On the Nielsen realisation problem for cyclic groups of prime order}
\author{ \large\textsc{Christian} KREMER}

\date{\today}

\begin{document}
	\maketitle
	
\begin{abstract}
	We solve the Nielsen realisation problem for high-dimensional aspherical manifolds in a new class of special cases. Mainly, we focus on actions of cyclic groups of prime order. A novelty is that it gives topological solutions to the problem with certain high-dimensional fixed point sets, whereas previous solutions of this type were restricted to discrete fixed point sets. This is enabled by recent joint work with Kirstein on isovariant Poincar\'e duality spaces, and Farrell--L\"uck--Steimle's obstruction theory to finding approximate fibrations within a homotopy class.
\end{abstract}

\section{Introduction}

Nielsen realisation problems in manifold theory are about refining symmetries of manifolds to stronger symmetries. Jakob Nielsen \cite{Nielsen32} originally asked if a finite subgroup of the mapping class group of a surface can be refined to a group action. For example, if $f$ is a selfmap of a surface $\Sigma_g$ with $f^p$ homotopic to the identity, is it homotopic to a homeomorphism of order $p$? There are various ways to generalise Nielsen's problem to other classes of manifolds $M$. This article focuses on a generalisation for aspherical manifolds: closed, connected manifolds with contractible universal cover.  Examples of aspherical manifolds are closed Riemannian manifolds of nonpositive sectional curvature, by the celebrated Cartan--Hadamard theorem. 
To formulate a version of the Nielsen realisation problem, let us write the outer automorphism group of a discrete group $\pi$ by $\Out(\pi)$. If $M$ is an aspherical manifold, then $\Out(\pi_1(M))$ is identified with the group of homotopy classes of homotopy automorphisms of $M$. Hence, subgroups of $\Out(\pi_1(M))$ encode symmetries of $M$ in a weak sense.

\begin{introquestion}
	\label{quest:nielsen}
	Let $M$ be a closed aspherical manifold, such that its fundamental group $\pi$ has trivial centre. When can a finite subgroup
	\begin{equation}
		G \subset \Out(\pi)
	\end{equation}
	be refined to a $G$-action by homeomorphisms on $M$?
\end{introquestion}

The triviality of the centre of $\pi$ guarantees that $G \subset \Out(\pi_1(M))$ refines to an $E_1$-action of $G$ on $M$, i.e. a map $BG \rightarrow B\hAut(M)$. See e.g. \cite[p.3]{borel} for an elaboration. 
In particular, this allows us to speak of the homotopy fixed points $M^{hG}$ of the action.
We give a purely group-theoretic description of this space below \cref{eq:homotopy_fixed_points}. Previous topological approaches to \cref{quest:nielsen}, for example \cite{dl5}, relied on these homotopy fixed points to be \textit{equivalent to a finite set}. This article gives the first purely topological construction of solutions to the Nielsen realisation problem in the case when $M^{hG}$ is \textit{low-dimensional compared to the dimension of $M$}.

\begin{introtheorem}
	\label{thm:main_theorem}
	Let $M^d$ be an oriented aspherical $d$-manifold with $\pi= \pi_1(M)$ word-hyperbolic and $d$ at least six. Consider a cyclic subgroup $C_p \subset \Out(\pi)$ of odd prime order $p$. If the space $M^{hC_p}$ is homotopy equivalent to a PL-manifold of dimension at most $(d-2)/2$,
	then it refines to a $C_p$-action on $M$.
\end{introtheorem}

The generalised Nielsen realisation problem (\cref{quest:nielsen}) has attracted some interest in the literature. In \cite{RaymondScott}, the importance of the triviality of the centre of $\pi$ was first discovered, see \cite{borel} for an elaboration.
The article \cite{BlockWeinberger} gives counterexamples for $G$ being a cyclic group of order $2$. Lastly, there are some uniqueness and non-uniqueness results for solutions to the Nielsen realisation problem \cite{CDK14, CDK15}. 

In fact, the action of $C_p$ on $M$ we produce satisfies a special property: It is \textit{Borel} in the sense that the map $M^{C_p} \rightarrow M^{hC_p}$ is an equivalence. This design criterion is attractive for the following reason: Let $\orbifoldfundamentalgroup \coloneqq \pi_1(M_{hC_p})$ denote the orbifold fundamental group of $(M, C_p \subset \Out(\pi))$ (see \cref{sec:setup}). That the action of $C_p$ on $M$ produced by \cref{thm:main_theorem} is Borel guarantees that $M$ is equivariantly homotopy equivalent to the quotient $\pi \backslash E_\finite \orbifoldfundamentalgroup$ with its residual $C_p$-action. Here, $E_\finite \orbifoldfundamentalgroup$ is the universal space for the family of finite subgroups of $\orbifoldfundamentalgroup$. Given a Borel $C_p$-action on $M$ as produced by \cref{thm:main_theorem}, the group $\orbifoldfundamentalgroup$ acts on the universal cover $\widetilde{M}$, exhibiting a \textit{cocompact manifold model for $E_\finite \orbifoldfundamentalgroup$}.

The space $M^{hC_p}$ can be identified explicitly in terms of the group $\orbifoldfundamentalgroup$. We show in \cref{appendixlemma:computing_fixed_points} that
\begin{equation}
	\label{eq:homotopy_fixed_points}
	M^{hC_p} \simeq \coprod_{F \in \maximal} BW_\orbifoldfundamentalgroup F.
\end{equation}
Here, $\maximal$ is a set of representatives of the nontrivial finite subgroups of $\orbifoldfundamentalgroup$ up to conjugation. The group $W_\orbifoldfundamentalgroup F = N_\orbifoldfundamentalgroup F/F$ is the Weyl group of $F$ in $\orbifoldfundamentalgroup$. The condition in \cref{thm:main_theorem} can hence be read as a condition on the Weyl groups of finite subgroups of the orbifold fundamental group $\orbifoldfundamentalgroup$. For example, if the normalisers of nontrivial finite subgroups of $\orbifoldfundamentalgroup$ are virtually cyclic, \cref{eq:homotopy_fixed_points} identifies $M^{hC_p}$ as a disjoint union of circles and points. We also mention that the indexing set $\maximal$ is a finite set, since a word-hyperbolic group has only finitely many conjugacy classes of finite subgroups \cite[Thm. III.$\Gamma$.3.2.]{BridsonHaefliger}.
The formula \cref{eq:homotopy_fixed_points} can also be used to see that it is in general not guaranteed for $M^{hC_p}$ to be homotopy equivalent to a closed manifold, see \cite{FL04}.

\subsection*{Outline of the proof of \cref{thm:main_theorem}}
	
The proof of \cref{thm:main_theorem} relies on \textit{equivariant and isovariant Poincar\'e duality}, two tools designed to describe facets of the homotopy theory of manifolds with a group action. Write $\nielsspace \coloneqq \pi \backslash E_\finite \orbifoldfundamentalgroup$ (see \cref{sec:setup} for an explanation of the notation). 

\begin{enumerate}[label=(\alph*)]
	\item Equivariant and isovariant Poincar\'e duality provide the structure of a \textit{finite semifree isovariant $C_p$-Poincar\'e space} on $\nielsspace$. That is, a decomposition of $\nielsspace$ as a pushout of finite $C_p$-spaces
	\begin{equation}
		\label{eq:diagram_introduction_isovariant}
		\begin{tikzcd}
			\partial C \ar[r] \ar[d] & C \ar[d]\\
			\nielsspace^{C_p} \ar[r] & \nielsspace
		\end{tikzcd}
	\end{equation}
	satisfying some geometric properties. Among those properties is that the action on $C_p$ and $\partial C_p$ is free, and that $(Q,\partial Q) = C_p\backslash(C,\partial C)$ is a Poincar\'e duality pair. Here, the codimension assumption in \cref{thm:main_theorem} is used.  Intuitively \cref{eq:diagram_introduction_isovariant} resembles the decomposition of a smooth $C_p$-manifold $M$ into its fixed point set $M^{C_p}$, the free part $M \setminus M^{C_p}$ and that they glue to $M$ along the unit sphere bundle of the normal bundle of the fixed point set \cite[Constr. 1.2]{semifree}. 
	\item We use non-equivariant surgery theory for topological manifolds, in particular Ranicki's total surgery obstruction to show that the pair $(Q,\partial Q)$ is equivalent to a compact manifold pair $(N,\partial N)$. The essential method here is a computation involving the $\ncKthy$-and $\Lthy$-theoretic Farrell--Jones conjecture. This is one step where the hyperbolicity assumption on $\pi$ is used.
	\item Picking an equivalence $\nielsspace^{C_p} \simeq L$, where $L$ is a PL-manifold, we want to glue it to the $C_p$-cover $\overline{N}$ of $N$ to obtain a manifold with $C_p$-action. For this, we use the technique of manifold approximate fibrations:  the teardrop construction allows us to build a topology on the $C_p$-set $(\overline{N}\setminus \partial \overline{N}) \cup L$ giving a $C_p$-manifold $M'$. As a last step, one shows that $M$ is homeomorphic to $M'$, compatible with the outer actions on the fundamental group. 
\end{enumerate}

Before moving on to how steps (a), (b) and (c) are carried out in this article specifically, we mention that also outside applications to the Nielsen realisation problem it presents a useful pipeline to construct $C_p$-actions on topological manifolds. First, one uses equivariant and isovariant Poincar\'e duality to separate the problem into the fixed point set part and the free part. Then one uses non-equivariant surgery to construct manifold structures on the free and fixed part separately. Finally, the theory of approximate fibrations is used to glue fixed part and free part back together.

Step (a) relies on and motivated the equivariant Poincar\'e duality developed in \cite{pd1,pd2} and on the theory of their isovariant structures built in \cite{semifree}. The work \cite{lueckBrown_Nielsen} is a precursor to these techniques, and heavily inspired them. 

Step (b) is mostly computational, partially generalising and streamlining computations in \cite{dl5} involving Farrell--Jones assembly maps. In contrast to \cite{dl5} our presentation consequently uses the modern setup of hermitian $K$-theory developed in \cite{nineauthors,nineauthorsii}, which might make the results therein accessible to a broader audience. We note some additional results made in the proof of this step below.

Step (c) is vastly simpler in the setting of discrete fixed points, than in our generality. Roughly, in the setting of discrete fixed points, the boundary $\partial \overline{N}$ will consist of spheres with a certain group action, and one can extend these group actions over a disc to obtain a group action on a closed manifold, which will in fact give the desired solution to the Nielsen realisation problem. For high-dimensional fixed points, it is not true in general that the map $\partial \overline{N} \rightarrow L$ is homotopic to an equivariant sphere bundle projection, so we cannot glue $\overline{N}$ to $L$ by gluing in a disc bundle. So we use the technique of manifold approximate fibrations and the teardrop construction from \cite{HughesTaylorWeinberger} to deal with this gluing construction. Since this technique is not broadly used, \cref{sec:Gluing} might be of some independent interest for readers interested in constructing group actions on topological manifolds.
We use the $\ncKthy$-theoretic obstruction theory to deform maps into approximate fibrations from \cite{FLS}. As this obstruction theory is only developed over PL-manifolds, we need the PL-assumption in \cref{thm:main_theorem}.

\subsection{Additional results}

In showing the existence of a manifold structure on the complement quotient $(Q,\partial Q)$ (i.e. the pair $C_p \backslash (C,\partial C)$ in \cref{eq:diagram_introduction_isovariant}) we prove some results on algebraic $K$-theory and Block homeomorphism groups that might be of independent interest. We start with the relevant $K$-theoretic results.
Let $G$ be a finite group, and $\pi$ a torsionfree group. Fix an extension
\begin{equation}
	1 \rightarrow \pi \rightarrow \Gamma \rightarrow G \rightarrow 1.
\end{equation}
A subgroup $F \subset \Gamma$ is called \textit{full} if the composite $F \subset \Gamma \rightarrow G$ is an isomorphism. Write $\maximal$ for a set of representatives of conjugacy classes of full subgroups. We assume that each nontrivial finite subgroup of $\Gamma$ is contained in a unique full subgroup. See \cref{appendixlemma:characterisations_of_property_M} for alternative characterisations of this property. The \textit{Whitehead group} $\whiteheadgroup(\Gamma)$ of $\Gamma$ is relevant to the theory of $s$-cobordisms, and the \textit{reduced projective class group} $\ncKthy(\bbZ \Gamma)$ of $\bbZ \Gamma$ to Wall's finiteness obstruction or Siebenmann's end theorem.
\begin{introthm}
	\label{thm:whitehead_groups}
	If $\Gamma$ is a Farrell--Jones group as above, then the inclusions $N_\Gamma F \subset \Gamma$ of normalisers of full subgroups of $\Gamma$  induce isomorphisms
	\begin{equation}
		\bigoplus_{F \in \maximal} \whiteheadgroup(N_\Gamma F) \xrightarrow{\simeq} \whiteheadgroup(\Gamma) \hspace{3mm} \text{and} \hspace{3mm} \bigoplus_{F \in \maximal} \widetilde{\ncKthy}_0(\bbZ N_\Gamma F) \xrightarrow{\simeq} \widetilde{\ncKthy}_0(\bbZ \Gamma). 
	\end{equation}
\end{introthm}
Examples of Farrell--Jones groups are word-hyperbolic groups or fundamental groups of nonpositively curved manifolds, and finite extensions of such. Having this in mind, we give an application of our results to the study of block homeomorphism groups. Let $M$ be a nonpositively curved closed manifold with an isometric action by the finite group $G$ of odd order. Assume that the action is semifree and smooth, so only $G$ and the trivial group occur as isotropy groups. We let $(\overline{N},\partial \overline{N})$ denote the complement of an equivariant tubular neighborhood of $M^G$ in $M$, and set $(N,\partial N) = G \backslash (\overline{N},\partial \overline{N})$. For these specific manifolds we get the following result for its group of block homeomorphisms.
\begin{introthm}
	\label{thm:block_borel}
	For $M$ as above, if $\dim M \geq 6 $ and $M^{G} \subset M$ has codimension at least $3$,
	the map
	\begin{equation}
		\label{eq:block_homeo_and_hAut}
		\widetilde{\mathrm{Homeo}}(N,\partial N) \rightarrow \hAut(N,\partial N)
	\end{equation}
	is an inclusion of path components onto the group of simple homotopy automorphisms.
\end{introthm}
The left hand side of \cref{eq:block_homeo_and_hAut} denotes the group of block homeomorphisms of $N$ respecting (but not fixing) the boundary, and the right hand side homotopy automorphisms of the pair $(N\partial N)$. A homotopy automorphism $(f,\partial f) \colon (N,\partial N) \rightarrow (N,\partial N)$ is simple if the maps $f$ and $\partial f$ are simple homotopy equivalences.  \cref{thm:block_borel} resembles the conclusion of the Block Borel conjecture for block homeomorphisms of aspherical manifolds. We view it as interesting, since most results for block homeomorphism groups of manifolds are known for manifolds with torsionfree fundamental group. See \cite{HLLRW21} for more information about the Block Borel conjecture, including its relevance to the vanishing results of characteristic classes.

\subsection{Recent literature on coherent actions}

We mention some recent work on the theory of \textit{coherent actions} on manifolds. If $M$ is a closed topological (resp. smooth) manifold and $G$ a compact Lie group, a coherent action on $M$ by homeomorphisms (resp. diffeomorphisms) is a map $\alpha \colon BG \rightarrow B\mathrm{Homeo}(M)$ (resp. $B\mathrm{Diff}$). Following Reinhold, we call such an action \textit{kinetic} if it comes from an actual (smooth) $G$-action on $M$. Reinhold \cite{Reinhold19} constructed non-kinetic smooth actions of the group $\mathrm{SU}(2)$ on high-dimensional manifolds, and recently Krannich--Randal-Williams constructed some examples of non-kinetic coherent actions by finite groups \cite{krw26}. The latter was inspired by a construction of Kang--Park--Taniguchi \cite{KPT26exotic}, who also gave a similar construction in dimension four \cite{KTP26nonkinetic}. In the results of this article, we start with a high-dimensional aspherical manifold $M$ with hyperbolic fundamental group and a map
\begin{equation}
	BC_p \rightarrow B\hAut(M) \simeq B\widetilde{\mathrm{Homeo}}(M)
\end{equation}
and lift it to an actual group action under the conditions of \cref{thm:main_theorem}. It would be interesting to investigate the relation of the theory of coherent actions with the methods used in this article, including equivariant or isovariant Poincar\'e duality, and equivariant surgery.

\subsection{Conventions}
We freely use the language of $\infty$-categories as developed by Joyal and Lurie. The term \textit{category} is used for an $\infty$-category. We write $\spc$ for the category of spaces, and $\spectra$ for the category of spectra. The term \textit{compact} is used in its categorical meaning, so a compact space is a $G$-space of the homotopy type of a finitely dominated $G$-CW complex. The exception being that ``compact manifold" is used for a manifold which is covering-compact. We sometimes write $e$ for the trivial group, so if $X$ is a $G$-space the notation $X^{e}$ refers to fixed points of the trivial group, i.e. the underlying space. For a discrete group $\Gamma$ and $F \subset \Gamma$ a subgroup, we write $N_\Gamma F$ for its normaliser and $W_\Gamma F = N_\Gamma F/F$ for its Weyl group.

\subsection{Acknowledgements}

I thank my advisor Wolfgang L\"uck for introducing me to the Nielsen realisation problem; this article is part of my PhD thesis. Further thanks go to Dominik Kirstein and Kaif Hilman for exploring equivariant Poincar\'e duality with me, which was crucial for this project. I thank Shmuel Weinberger for pointing out the relevance of approximate fibrations for equivariant manifold theory to me, and Andrea Bianchi, Emma Brink, Branko Juran, Markus Land as well as Marco Volpe for helpful conversations.

\tableofcontents

\section{Setup and decompositions from equivariant Poincar\'e duality}

\label{sec:setup}

In this subsection, we recall some results from equivariant and isovariant Poincar\'e duality that are crucial in the proof of \cref{thm:main_theorem}. Let $G$ be a finite group. For the purpose of our main theorem, $G$ will be $C_p$, but for most of the article there is no reason to make this restriction on the generality.
A space is \textit{aspherical} if it is connected and has contractible universal cover.
Assume that $X$ is an aspherical space with fundamental group $\pi$, so that $X \simeq B\pi$.  Let $\alpha \colon G \rightarrow \hAut(X)$ be a map of $E_1$-groups.  This can equivalently be encoded by an extension of discrete groups
\begin{equation}
	\label{eq:extension}
	1 \rightarrow \pi \rightarrow \orbifoldfundamentalgroup \rightarrow G \rightarrow 1.
\end{equation}
Indeed, given $\alpha$, we can set $\orbifoldfundamentalgroup \coloneqq \pi_1 (X_{hG})$ to be the \textit{orbifold fundamental group} of the action. The fibre sequence $X \rightarrow X_{hG} \rightarrow BG$ induces the extension \cref{eq:extension} on fundamental groups.
Conversely, in \cref{eq:extension} we can pass to classifying spaces to get a map $B\orbifoldfundamentalgroup \rightarrow BG$ with fibre $B\pi$, exhibiting a $G$-action on $B\pi$ by fibre transport.
As mentioned in the introduction, if $\pi$ has trivial centre, the $E_1$-group $\hAut(X)$ is equivalent to the discrete group $\Out(\pi)$. Hence, in this case the extension above can be recovered from the pair $(X,G \subset \Out(\pi))$.

We write $\nielsspace$ for the quotient $\pi \backslash E_\finite \orbifoldfundamentalgroup$. Viewing $X \in \spc^{BG}$ as a space with $G$-action, $\nielsspace$ is the image of $X$ under the right adjoint of the restriction $\spc_{G} \rightarrow \spc^{BG}$, see \cite{borel}. This right adjoint is fully faithful with image the \textit{Borel} $G$-spaces, whence the name.

The main results of this section are about decompositions of $\nielsspace$ needed in the rest of the article. 
Crucially, to prove \cref{thm:main_theorem} we will have to decompose $\nielsspace$ into its fixed part $\nielsspace^{C_p}$ and a complement $C$, intuitively thought of as the complement $M \setminus M^{C_p}$ of the fixed points for a manifold $M$ with $C_p$-action. This decomposition is additional data on $\nielsspace$ and might not exist in general.

\subsection{Good decompositions and isovariant Poincar\'e spaces}

In this subsection we define \textit{good decompositions} of $G$-spaces $X$ as a weaker version of semifree isovariant $G$-Poincar\'e spaces that we encounter thereafter. The reader should think of the free-fixed decomposition of a smooth semifree $G$-manifold, see \cite[Constr. 1.2]{semifree}. 

Recall that a map $f \colon A \rightarrow X$ of spaces is a \textit{$\pi$--$\pi$-equivalence} if it induces an equivalence on $1$-truncations. Equivalently, the map $\pi_0 f \colon \pi_0 A \rightarrow \pi_0 X$ is a bijection, and for each choice of basepoint $a \in A$ the map $\pi_1(f,a) \colon \pi_1(A,a) \rightarrow \pi_1(X,f(x))$ on fundamental groups is an equivalence.

\begin{defn}
	A \textit{good decomposition} of a compact $G$-space $X$ is a pushout square
	\begin{equation}
		\label{eq:diagram_setup_isovariant}
		\begin{tikzcd}
			\partial C \ar[r] \ar[d, "p"] & C \ar[d]\\
			X^{G} \ar[r, "\mathrm{incl}"] & X
		\end{tikzcd}
	\end{equation}
	satisfying the following conditions.
	\begin{enumerate}
		\item The $G$-action on $C$ and $\partial C$ is free.
		\item The pair $(C^e, \partial C^e)$ is a Poincar\'e duality pair.
		\item The map $\partial C^{e} \rightarrow X^{G}$ is a $\pi$--$\pi$-equivalence.
		\item For $x \in X^{G}$, the fibre $\{x\} \times_{X^{G}} \partial C$ is a compact $G$-space.
	\end{enumerate}
	A good decomposition of $X$ is \textit{finite} if all the spaces occurring in \cref{eq:diagram_setup_isovariant} are finite $G$-spaces. 
\end{defn}

\begin{rmk}
	If $X$ admits a good decomposition, and $e \neq H \subset G$ is any subgroup, then the map $X^{G} \rightarrow X^{H}$ is an equivalence. In  other words, $X$ is a \textit{semifree $G$-space}. The main results of this article are for the group $C_p$, where this condition is vacuous. 
\end{rmk}

\begin{defn}
	\label{def:complement_quotient}
	Given a good decomposition of a compact $G$-space $X$ as in \cref{eq:diagram_setup_isovariant}, we call $(C,\partial C)$ the \textit{complement} (of $X^G$ in $X$) and $(Q,\partial Q) \coloneqq G \backslash (C,\partial C)$ the \textit{complement-quotient}.
\end{defn}

One part of this article is about showing that given a good decomposition for $\nielsspace$ as in \cref{eq:diagram_setup_isovariant}, there is a compact manifold pair equivalent to the pair $(Q,\partial Q) \coloneqq G \backslash (C,\partial C)$. Or equivalently, there is a compact manifold pair $(\overline{N},\partial \overline{N})$ within the $G$-homotopy type $(C,\partial C)$. But to find a closed $G$-manifold within the $G$-homotopy type $\nielsspace$, we actually need the following stronger notion.

\begin{defn}[\cite{semifree}, Def. 2.3.4.]
	\label{defn:semifree_isovariant_structure}
	A compact $G$-space $X$ with a good decomposition as in \cref{eq:diagram_setup_isovariant} is a \textit{semifree isovariant $G$-Poincar\'e duality space} if $X^{G}$ is a Poincar\'e duality space, and for each $x \in X^{G}$ the fibre $\partial C^{e} \times_{X^{G}} \{x\}$ is equivalent to a sphere $S^l$ of some dimension $l$. If the good decomposition of $X$ is additionally finite, it is called a \textit{finite isovariant $G$-Poincar\'e duality space}. 
\end{defn}

\subsection{Finiteness}

The homotopy types of compact manifolds are \textit{finite} spaces, i.e. spaces admitting a finite cell structure. Given a $G$-space with a good decomposition \cref{eq:diagram_introduction_isovariant}, checking finiteness of $C$ and $\partial C$ as $G$-spaces is a first obstruction to finding a compact $G$-manifold pair $(\overline{N},\partial \overline{N})$ in the $G$-homotopy type $(C,\partial C)$.
\begin{recollect}
	\label{recollect:compact_implies_finite}
	If $X$ is a compact $C_p$-space with free action and $X^{e} \simeq S^l$ for some $l$, then $C_p \backslash X$ is even finite, because it is equivalent to some lens space.
	See \cite[Sec. 2.3]{Nicholson24} for more groups with similar properties, e.g. also a compact free $\mathrm{SL}_2(\mathbb{F}_p)$-space $X$ with $X^{e} \simeq S^l$ has the property that $G \backslash X$ is finite for $p \geq 3$.
\end{recollect}

\begin{lem}
	\label{lem:finiteness_for_good_decompositions}
	Let consider a compact $G$-space $X$ with a good decomposition
	\begin{equation}
		\begin{tikzcd}
			\partial C \ar[r] \ar[d, "p"'] & C \ar[d, "p'"]\\
			X^{G} \ar[r] & X.
		\end{tikzcd}
	\end{equation}
	Then:
	\begin{enumerate}
		\item Finiteness of $C$ is implied by finiteness of $X$ and $\partial C$.
		\item If $G=C_p$ and the square is a $G$-isovariant Poincar\'e space, then finiteness of $\partial C$ is implied by finiteness of $X^{G}$.
	\end{enumerate}
\end{lem}

\begin{recollect}
	To prove \cref{lem:finiteness_for_good_decompositions}, we use L\"uck's equivariant finiteness obstruction, a generalisation of Wall's $K$-theoretic finiteness obstruction \cite[Sec. 14]{lueck_transformation}. Let $\abeliangroups$ denote the category of abelian groups. There is a functor
	\[ \equivprojclass{G}{-} \colon \spc_G \rightarrow \abeliangroups.  \]
	Whenever $X$ is a compact $G$-space, there is a well-defined element $\equivfiniteness(X) \in \equivprojclass{G}{X}$, which vanishes exactly if $X$ has the homotopy type of a finite $G$-CW complex. 
	The element $\equivfiniteness(X) \in \equivprojclass{G}{X}$ is called the \textit{equivariant finiteness obstruction} of $X$. Given a pushout of compact $G$-spaces
	\begin{equation}
		\begin{tikzcd}
			X \ar[r] \ar[d] \ar[dr, "h"] & X' \ar[d, "i"] \\
			Y \ar[r, "j"] & Y'
		\end{tikzcd}
	\end{equation}
	the equivariant finiteness obstructions satisfy the sum formula
	\begin{equation}
		\equivfiniteness(Y') = \equivprojclass{G}{i} (\equivfiniteness(X')) + \equivprojclass{G}{j} (\equivfiniteness(Y)) - \equivprojclass{G}{h} (\equivfiniteness(X)).
	\end{equation}
	If $X$ is compact, L\"uck constructs \cite[Thm. 14.46]{lueck_transformation} an explicit natural isomorphism
	\begin{equation}
		\label{eq:computation_of_equivariant_proj_class}
		\bigoplus_{(H)}  \reducedK_0( X^H_{hW_G H} ) \xrightarrow{\simeq} \equivprojclass{G}{X}.
	\end{equation}
	Here, the indexing set is a set of representatives of conjugacy classes of subgroups of $G$.
\end{recollect}

\begin{proof}[Proof of \cref{lem:finiteness_for_good_decompositions}]
	For the first part, we have to show that finiteness of $X$ and $\partial C$ implies finiteness of $C$. From the sum formula for the equivariant finiteness obstruction, we see that the image of $\equivfiniteness(C) \in \equivprojclass{G}{C}$ in $\equivprojclass{G}{X}$ is zero. Hence, we have to argue that the map
	\begin{equation}
		\equivprojclass{G}{p'} \colon \equivprojclass{G}{C} \rightarrow \equivprojclass{G}{X}
	\end{equation}
	is injective. From the natural isomorphism \cref{eq:computation_of_equivariant_proj_class}, it suffices to show that $\reducedK_0(C^{e}_{hG}) \rightarrow \reducedK_0(X^{e}_{hG})$ is injective. But this map is even an isomorphism: since $\partial C^{e} \rightarrow X^{G}$ is a $\pi$--$\pi$-equivalence, so is $C^{e} \rightarrow X^{e}$, and $\reducedK_0(-)$ inverts $\pi$--$\pi$-equivalences.
	To prove the second assertion, note that by \cref{recollect:compact_implies_finite} the space $\partial Q = C_p \backslash \partial C^{e}$ can be obtained as a colimit over the finite complex $X^G$ of finite complexes, and hence is itself finite.
\end{proof}

\subsection{Existence of finite isovariant structures for $\nielsspace$}

In this subsection, we give the main existence theorem on additional structure for $\nielsspace$ that we use to prove \cref{thm:main_theorem}. 

\begin{thm}
	\label{thm:existence_of_finite_isovariant_structures}
	Suppose given an aspherical manifold $M^d$ with word-hyperbolic fundamental group $\pi$, and a subgroup $C_p \subset \Out(\pi)$. Assume that $M^{hC_p}$ is a Poincar\'e duality space of dimension at most $(d-2)/2$. Then the associated $C_p$-space $\nielsspace$ admits the structure of a finite isovariant $C_p$-Poincar\'e space.
\end{thm}

\begin{proof}
	In \cite{semifree}, it was proven that $\nielsspace$ admits the structure of an isovariant $C_p$-Poincar\'e space, and we check that it is indeed finite. The orbifold fundamental group $\orbifoldfundamentalgroup = \pi_1 M_{hC_p}$ is again word-hyperbolic, since it is commensurable to $\pi$. A Rips complex of $\orbifoldfundamentalgroup$ presents a finite cell structure on $E_\finite \orbifoldfundamentalgroup$. Since $\nielsspace \simeq \pi \backslash E_\finite \orbifoldfundamentalgroup$ we conclude that $\nielsspace$ is a finite $C_p$-space. This implies that $\nielsspace^{C_p} \simeq M^{hC_p}$ is finite as well. Now \cref{lem:finiteness_for_good_decompositions} implies that every isovariant structure on $\nielsspace$ is indeed finite.
\end{proof}

\begin{rmk}	
	The quoted results of \cite{semifree} rely on \cite{pd2}, which was only written in the generality of cyclic groups of prime order. It seems plausible that the results of \cite{pd2} can be generalised, extending the validity of \cref{thm:existence_of_finite_isovariant_structures}.
	To account for this possibility, we often require a good decomposition (or isovariant structure, etc.) of $\nielsspace$ as a datum instead of making the assumptions of \cref{thm:main_theorem} and citing \cref{thm:existence_of_finite_isovariant_structures} in the rest of this article.
\end{rmk}

\section{Surgery-theoretic preliminaries}

In this preliminary section, we recall algebraic $K$- and $L$-theory in the framework of Poincar\'e categories introduced in \cite{nineauthors}. We give a construction of the Poincar\'e category relevant to the total surgery obstruction. This Poincar\'e category is constructed covariantly functorial in maps of pairs of spaces, and contravariantly functorial in covering maps. We review the total surgery obstruction of a Poincar\'e duality pair, whose vanishing we want to verify in our special case in order to prove \cref{thm:main_theorem}. 

\subsection{$K$- and $L$-theory}
We use the theory of \textit{Poincar\'e categories} as developed in the series \cite{nineauthors,nineauthorsii}. Write $\catp$ for the category of Poincar\'e categories, which are pairs
\[ (\category{C}, \qoppa) \in \catp \]
with $\category{C}$ a stable category and $\qoppa \colon \category{C}\op \rightarrow \spectra$ a perfect quadratic functor. There is a functor $\Poinc \colon \catp \rightarrow \spc$ assigning to a Poincar\'e category its space of \textit{Poincar\'e objects}. Intuitively: a Poincar\'e category $(\category{C},\qoppa)$ is a category of algebraic objects $\category{C}$ and a type of forms $\qoppa$. A Poincar\'e object is a pair $(x,q)$ where $x \in \category{C}$ is an object, and $q \in \qoppa(x)$ is a ``non-degenerate form of type $\qoppa$" on it. For a ring $R$, write $\module{R}^\finitemodule$ for the smallest stable subcategory of the (derived) module ($\infty$--)category $\module{R}$ generated by $R$.  Objects in $\module{R}^\finitemodule$ are called \textit{finite} $R$-modules.
\begin{example}
	There is a Poincar\'e category given by the pair $(\module{\bbZ}^\finitemodule,\quadratic)$ where $\quadratic$ is the functor
	\[ \quadratic(x) \coloneqq \hom(x \otimes x, \bbZ)_{hC_2}. \]
	A Poincar\'e object in $(\module{\bbZ}^\finitemodule,\quadratic)$  is a finite $\bbZ$-module $x$ carrying a non-degenerate quadratic form.
\end{example}

Ranicki's work on algebraic surgery and the subsequent redevelopment and generalisation in \cite{nineauthors, nineauthorsii} identifies many objects encountered in surgery theory in terms of \textit{Poincar\'e-Verdier invariants} of Poincar\'e categories. We recall a few such invariants here to fix our notation for the rest of the article.
\begin{enumerate}
	\item \textit{Connective algebraic $K$-theory} is a functor
	\[ \Kthy \colon \catp \rightarrow \spectra^{BC_2}_{\geq 0} \subset \spectra^{BC_2}  \]
	which receives a natural transformation $\Poinc \rightarrow \Omega^{\infty}\Kthy(-)^{hC_2}$.
	\item \textit{Non-connective algebraic $K$-theory} is a functor
	\[ \ncKthy \colon \catp \rightarrow \spectra^{BC_2} \]
	which comes with a natural transformation $\Kthy \rightarrow \ncKthy$. This natural transformation does \textit{not} become an equivalence when postcomposed with the connective cover functor $\tau_{\geq 0}$. However, it does induce an equivalence $\tau_{\geq 1} \Kthy \rightarrow \tau_{\geq 1} \ncKthy$.
	\item Postcomposing the map $\Kthy \rightarrow \ncKthy$ with the Tate fixed-point functor $(-)^{tC_2}$ gives rise to two new Poincar\'e-Verdier invariants and a comparison map between them
	\[ \left(\tateK \rightarrow \nctateK \right) \colon \catp \rightarrow \spectra.  \]
	\item \textit{Algebraic $L$-theory} is a functor
	\[ \Lthy \colon \catp \rightarrow \spectra \]
	which serves as a broad generalisation and refinement of Wall's $L$-groups of rings with an anti-involution. In special cases, its homotopy groups capture invariants of symmetric bilinear or quadratic forms, and as such occur as obstruction groups in surgery theory.
	\item \textit{Karoubi-invariant algebraic $L$-theory} is a functor
	\[ \karLthy \colon \catp \rightarrow \spectra. \]
	It comes with a natural transformation $\Lthy \rightarrow \karLthy$ which can be described by a universal property:
	while algebraic $L$-theory does not invert Karoubi equivalences \cite[Def. 1.3.1.]{nineauthorsii}, Karoubi-invariant algebraic $L$-theory does, and this transformation exhibits it as the initial Karoubi-invariant approximation of algebraic $L$-theory from the right among Poincar\'e-Verdier invariants.
\end{enumerate}
Our computational goal can be stated purely in terms of algebraic $L$-theory. However, Karoubi-invariant algebraic $L$-theory has significant computational advantages over its non--Karoubi-invariant counterpart, such as the Farrell--Jones conjecture as we recall in \cref{subsec:Davis_Lueck_assembly}. Algebraic $K$-theory and its non-connective counterpart will primarily be used to get computational access to the difference of $\Lthy$ and $\karLthy$. The following result is from the forthcoming article \cite{nineauthorsiv}.

\begin{prop}
	\label{prop:karoubi_invariant_L_theory_computation}
	The natural transformations $\Lthy \rightarrow \karLthy$ and $\tateK \rightarrow \nctateK$ fit into a commuting square
	\begin{equation}
		\begin{tikzcd}
			\Lthy \ar[r] \ar[d] & \tateK \ar[d]\\
			\karLthy \ar[r] & \nctateK
		\end{tikzcd}
	\end{equation}
	which induces a pullback of spectra when evaluated at any Poincar\'e category $(\category{C},\qoppa)$.
\end{prop}

We also recall the notion of \textit{metabolic Poincar\'e categories} and \textit{bordism invariant Poincar\'e-Verdier invariants} for later use. 
Recall from \cite[Sec. 2.3.]{nineauthors} the metabolic construction
\[ \Met \colon \catp \rightarrow \catp, \hspace{3mm} (\category{C},\qoppa) \mapsto (\Ar(\category{C}),\qoppamet). \]
Here, $\Ar(\category{C})$ is the arrow category of $\category{C}$, and $\qoppamet(w \rightarrow x) \coloneqq \fib(\qoppa(x) \rightarrow \qoppa(w))$. There is a Poincar\'e functor $\mathrm{met} \colon \Met(\category{C},\qoppa) \rightarrow (\category{C},\qoppa)$, which on objects extracts the target of an arrow $w \rightarrow x$. Poincar\'e objects in $\Met(\category{C},\qoppa)$ admit an interpretation in terms of Lagrangians of forms. Examples of such Lagrangians are provided by nullbordisms of manifolds \cite[Ex. 2.3.4.]{nineauthors}.
\begin{obs}
	By \cite[Rmk. 7.3.21]{nineauthors}, the Poincar\'e functor $\mathrm{met}$ is an adjunction counit, with the right adjoint being the forgetful functor from Poincar\'e categories to hermitian categories. This forgetful functor commutes with colimits \cite[Prop. 6.1.4.]{nineauthors}, so $\mathrm{Met}$ is colimit-preserving.
\end{obs}
A Poincar\'e-Verdier invariant is called \textit{bordism invariant} if it vanishes on $\Met(\category{C},\qoppa)$ for every Poincar\'e category $(\category{C},\qoppa) \in \catp$. See \cite[Lem. 3.5.6.]{nineauthorsii} for a few equivalent characterisations. 

\begin{prop}
	The Poincar\'e-Verdier functors $\Lthy, \karLthy, \tateK$ and $\nctateK$ are bordism invariant.
\end{prop}

\begin{proof}
	This is \cite[Cor. 4.4.6.]{nineauthorsii} for $\Lthy$, \cite[Ex. 3.3.6.(i)]{nineauthorsii} for $\tateK$ and $\nctateK$, and in the upcoming article \cite{nineauthorsiv} for $\karLthy$.
\end{proof}

\subsection{$L$-theory with transfers}

\label{subsec:L_thy_with_trf}

The category $\catp$ of Poincar\'e categories has small limits and colimits \cite[Prop. 6.1.4.]{nineauthors}. Given a space $X \in \spc$ and a Poincar\'e category $(\category{C},\qoppa)$, we write $(\category{C}.\qoppa)_X$ for the colimit of the constant $X$-shaped diagram with value $(\category{C},\qoppa)$. This construction is evidently covariantly functorial in $X$. In the following, we describe a related construction for \textit{pairs of spaces} and functoriality in \textit{finite sheeted covering maps}.

Given a pair of spaces $(X,A)$ and $(\category{C},\qoppa) \in \catp$, we define $(\category{C},\qoppa)_{(X,A)}$ as a pullback in $\catp$.
\begin{equation}
	\label{eq:relative_metabolic_construction}
	\begin{tikzcd}
		(\category{C},\qoppa)_{(X,A)} \ar[r] \ar[d] & \Met((\category{C},\Sigma \qoppa)_X) \ar[d, "\mathrm{met}"]\\
		(\category{C},\Sigma \qoppa)_A \ar[r] & (\category{C},\Sigma \qoppa)_X
	\end{tikzcd}
\end{equation}
Since the formation of Poincar\'e objects commutes with pullbacks, we may view Poincar\'e objects in $(\category{C},\qoppa)_{(X,A)}$ as a Poincar\'e object in $(\category{C},\Sigma \qoppa)_A$ together with a nullbordism of its pushforward to $(\category{C},\Sigma \qoppa)_X$. Intuitively, it resembles a class in the relative bordism group $\Omega_*(X,A)$.
We make the following observations.
\begin{enumerate}
	\item Since the right vertical morphism in the pullback \cref{eq:relative_metabolic_construction} is a split Poincar\'e-Verdier projection (\cite[Ex. 1.2.5.]{nineauthorsii}), the square is a split Poincar\'e-Verdier square (\cite[Rem. 1.6.2.iii]{nineauthorsii}).
	\item In particular, since $L$-theory is bordism invariant the square \cref{eq:relative_metabolic_construction} induces a fibre sequence
	\[ \Lthy((\category{C},\qoppa)_A) \rightarrow \Lthy((\category{C},\qoppa)_X) \rightarrow  \Lthy( (\category{C},\qoppa)_{(X,A)} ) . \]
	An analogous observation holds for the functors $\karLthy, \tateK$ and $\nctateK$.
	\item The construction is covariantly functorial in $(X,A) \in \spc^2$.
\end{enumerate}

In the following, we extend the functoriality of the construction \cref{eq:relative_metabolic_construction} to backwards-functoriality in finite covering maps. 
Write $\spc^2 \coloneqq \func(\Delta^1,\spc)$ for the category of pairs of spaces. A map $p \colon (Y,\partial Y) \rightarrow (X,\partial X)$ is a \textit{finite covering of pairs} if the square
\begin{equation}
	\begin{tikzcd}
		\partial Y \ar[r] \ar[d] & Y \ar[d]\\
		\partial X \ar[r] & X
	\end{tikzcd}
\end{equation}
is cartesian and both vertical maps are covering maps. Using the notation of \cref{sec:pairs_of_spaces}, we define the category of \textit{pairs of spaces with transfers} as the category of correspondences
\[ \spcpairtrans \coloneqq \Corr(\spc^2,\fincov, \all). \]
Objects in $\spcpairtrans$ are pairs of spaces, and morphisms  are those correspondences
\begin{equation}
	\label{eq:generic_span_with_finite_cover}
	\begin{tikzcd}
		& (E,\partial E) \ar[dl, "p"'] \ar[dr, "f"] &\\
		(X,\partial X) && (Y,\partial Y)
	\end{tikzcd}
	\in \map_{\Corr(\spc^2)}((X,\partial X),(Y,\partial Y))
\end{equation}
for which $p$ is a finite-sheeted covering map of pairs of spaces. Composition is defined by taking pullbacks of correspondences, which is well-defined since the class of finite coverings of pairs is stable under pullbacks.

\begin{constr}
	Let $\category{J} = (\ast \rightarrow \ast \leftarrow \ast)$ denote the free cospan . The category $\func(\category{J}, \catp)$ is semiadditive and cocomplete. Note that assigning to $(X,A) \in \spc^2$ the span $(\category{C},\Sigma \qoppa)_A \rightarrow (\category{C},\Sigma\qoppa)_X \leftarrow \Met((\category{C},\Sigma \qoppa)_X)$ is a colimit preserving functor in $(X,A)$. Hence, by \cref{prop:constructing_functors_by_describing_map} it uniquely extends to a functor $u \colon \spcpairtrans \rightarrow \func(\category{J},\catp)$. The functor
	\begin{equation}
		(\category{C},\qoppa)_{(-)} \colon \spcpairtrans \rightarrow \catp
	\end{equation}
	is the composite of the functor $u$ with the pullback functor $\func(\category{J},\catp) \rightarrow \catp$. Note that the restriction of this functor to $\spc^2$ recovers the functor written in \cref{eq:relative_metabolic_construction}.
\end{constr}

\subsection{$L$-theory for pairs of spaces and structure spectra}

\label{subsec:Lthy_for_pairs}

Next, we introduce the main example of a Poincar\'e category that we concern ourselves with. 

\begin{defn}
	Let $(X,A)$ be a pair of spaces. Its \textit{quadratic $L$-theory} is the spectrum
	\[ \quadL(X,A) = \Lthy((\module{\bbZ}^\finitemodule,\quadratic)_{(X,A)}). \]
	It defines a functor $\quadL \colon \spcpairtrans \rightarrow \spectra$. We abbreviate $\quadL(X) \coloneqq \quadL(X,\emptyset)$. Analogous constructions with $\karLthy, \tateK,\nctateK$ define functors $\quadkarL, \tateK,\nctateK \colon \spcpairtrans \rightarrow \spectra$.
\end{defn}

Note that, as discussed in \cref{subsec:L_thy_with_trf},  $\quadL(X,A)$ is the cofiber of the map $\quadL(A) \rightarrow \quadL(X)$ induced by the functoriality of $\quadL(-)$.  Further, for a discrete group $\Gamma$, $\quadL(B\Gamma) = \Lthy( \module{\bbZ \Gamma}^\finitemodule, \quadratic )$ \cite[p. 1489]{LandNikolausSchlichting23}. 

\begin{rmk}
	See \cite[Cor. 4.4.12]{nineauthors} for a discussion relating the present definition of quadratic $L$-theory to previously considered definitions. In \cite[Sec. 1.2]{nineauthorsiii} the authors derive a presentation of the homotopy groups of quadratic $L$-theory that is akin to Wall's original definition \cite{scm}. This is achieved by adapting Ranicki's algebraic surgery to the setting of Poincar\'e categories. Wall originally defined quadratic $L$-groups for a \textit{groupoid} instead of a space. From our perspective, this is explained by the following important property of quadratic $L$-theory.
\end{rmk}

\begin{prop}[The algebraic $\pi$--$\pi$-theorem, \cite{nineauthorsiii} Cor. 1.2.33.]
	\label{prop:algebraic_pi_pi_theorem}
	Suppose $f \colon A \rightarrow X$ is a $\pi$--$\pi$-equivalence. Then $\quadL(X,A) \simeq 0$.
\end{prop}

Our next goal is to construct the assembly maps in quadratic $L$-theory compatible with transfers, and a slight variant needed to define the structure spectrum of a pair of spaces. We start by a general construction.

\begin{constr}
	Let $E \in \spectra$ be a spectrum. Using \cref{prop:constructing_functors_by_describing_map} we can construct a functor
	\[ E \otimes (-) \colon \spcpairtrans \rightarrow \spectra \]
	by assigning the values
	\[ E\otimes (\ast,\emptyset) \coloneqq E \hspace{3mm} \text{and} \hspace{3mm} E \otimes(\ast,\ast) \coloneqq 0. \]
	In fact, by \cref{prop:functor_out_of_pairs_with_transfers_satisfying_codescent} a functor $F \in \func^\codesc(\spcpairtrans,\spectra)$ is of the form $E \otimes (-)$ if and only if $F(\ast,\ast) \simeq 0$, in which case $F \simeq F(\ast,\emptyset) \otimes (-)$. 
\end{constr}

\begin{defn}
	\label{defn:approximating_functor_from_left_by_excisive}
	Let $F \in \func(\spcpairtrans,\spectra)$ be a functor commuting with coproducts. Assume $F(\ast,\ast) \simeq 0$. Then we write
	\begin{equation}
		\asm_F \colon F(\ast) \otimes (-) \rightarrow F(-)
	\end{equation}
	for the adjunction counit of the adjunction \cref{prop:functor_out_of_pairs_with_transfers_satisfying_codescent}, evaluated on $F$. We call $\asm_F$ the \textit{assembly map with transfers} for the functor $F$. 
\end{defn}

\begin{constr}[Assembly maps]
	\label{constr:assembly}
	Observe that $\quadL(-)$ commutes with coproducts, i.e. the map
	\[ \quadL(X_0,A_0) \oplus \quadL(X_1,A_1) \rightarrow \quadL(X_0 \cup X_1,A_0 \cup A_1)  \]
	is an equivalence.
	The restriction of $\quadL(-)$ to $\spc^2$ does not commute with colimits, however. We have shown in \cref{prop:functor_out_of_pairs_with_transfers_satisfying_codescent} that the inclusion $\func^\codesc(\spcpairtrans,\spectra) \subset \func^\cup(\spcpairtrans,\spectra)$ admits a right adjoint. Hence, there is a universal natural transformation
	\[ \asm_{\quadL} \colon \quadL(\ast) \otimes (-) \rightarrow \quadL(-) \]
	whose source lies in $\func^\codesc(\spcpairtrans,\spectra) $. To justify notation, observe that the restriction of $\quadL$ to $\spc^2$ may be computed by left Kan extending the restriction of $\quadL(-)$ along $\Delta^1 \subset \spc^2 \rightarrow \spcpairtrans$ back to $\spc^2$. We compute
	\[ \quadL(\ast) \coloneqq \quadL(\emptyset \rightarrow *) \simeq \Lthy(\module{\bbZ}^{\finitemodule},\quadratic) \hspace{3mm} \text{and} \hspace{3mm} \quadL(* \rightarrow *) \simeq \Lthy(\Met(\module{\bbZ}^{\finitemodule},\quadratic)) \simeq 0. \]
	In general $\quadL(\ast) \otimes (X,A)$ is the cofibre of the map $\quadL(\ast) \otimes \Sigma^\infty_+ A \rightarrow \quadL(\ast) \otimes \Sigma^\infty_+ X$. Similarly $\quadkarL$ admits such an approximation from the left written $\asm_{\quadkarL} \colon \quadkarL(\ast) \otimes (-) \rightarrow \quadkarL(-)$. The map $\quadL(-) \rightarrow \quadkarL(-)$ induces a commuting square of functors
	\begin{equation}
		\begin{tikzcd}
			\quadL(\ast) \otimes (-) \ar[r, "\simeq"] \ar[d, "\asm_{\quadL}"] & \quadkarL(*) \otimes (-) \ar[d,"\asm_{\quadkarL}"]\\
			\quadL(-) \ar[r] & \quadkarL(-)
		\end{tikzcd}
	\end{equation}
	in which the upper map is an equivalence: indeed, this follows from the map
	\[ \Lthy(\module{\bbZ}^{\finitemodule},\quadratic) \rightarrow \karLthy(\module{\bbZ}^{\finitemodule},\quadratic) \]
	being an equivalence.
\end{constr}

We write $\tau_{\geq 1}(-)$ for the $1$-connective cover functor, viewed as an endofunctor of $\spectra$. Note that it comes with a map $\tau_{\geq 1} E \rightarrow E$, natural in $E \in \spectra$.

\begin{defn}
	Let $(Y,A)$ be a pair of spaces. The \textit{algebraic structure spectrum} of $(Y,A)$ is the spectrum
	\begin{equation}
		\structure(Y,A) \coloneqq \cofib \left( (\connquadL(\ast)) \otimes (Y,A) \rightarrow \quadL(\ast) \otimes (Y,A) \rightarrow \quadL(Y,A) \right).
	\end{equation}
\end{defn}
We warn the reader that only if $(X,\partial X)$ is an \textit{oriented} Poincar\'e duality pair, the spectrum $\structure(X,\partial X)$ has a relation to compact manifold pairs in the homotopy type $(X,\partial X)$. For non-orientable Poincar\'e duality pairs, one uses a twisted version \cite[Appendix A]{Ranicki92}. 
For computational reasons, we will also consider the following two variants. Define
\begin{align*}
	\structurekar(Y,A) \coloneqq \cofib( (\connquadkarL(\ast) \otimes (Y,A)) \rightarrow \quadkarL(Y,A) ) & \hspace{5mm} \text{the \textit{Karoubi structure spectrum;}}\\
	\structurekarper(X,A) \coloneqq \cofib(\quadkarL(\ast) \otimes (Y,A) \rightarrow \quadkarL(Y,A)) & \hspace{5mm}  \text{ the \textit{periodic Karoubi structure spectrum.} }
\end{align*}
All three variants $\structure(-), \structurekar(-), \structurekarper(-)$ are functorial in $\spcpairtrans$.

\begin{prop}
	\label{prop:karoubi_invariance_for_structure_spectra}
	The fibres of the maps $\structure(X,A) \rightarrow \structurekar(X,A)$ and $\tateK(X,A) \rightarrow \nctateK(X,A)$  are equivalent, naturally in $(X,A) \in \spcpairtrans$.
\end{prop}

\begin{proof}
	Note that from the definition of $\structure(-)$ as a cofibre it comes with a natural transformaion from $\quadL(-)$. The two squares in the diagram
	\begin{equation}
		\begin{tikzcd}
			\structure(X,A)  \ar[d] & \quadL(X,A) \ar[r] \ar[l] \ar[d] & \tateK(X,A) \ar[d]\\
			\structurekar(X,A)  & \quadkarL(X,A) \ar[r] \ar[l]& \nctateK(X,A)
		\end{tikzcd}
	\end{equation}
	induce a functorial span between the desired fibres. We argue that this span is in fact an equivalence. To see this, we have to argue that the two squares are cartesian. The right square is cartesian by \cref{prop:karoubi_invariant_L_theory_computation}. The map induced on horizontal fibres by the left square is the map $(\connquadL(\ast)) \otimes (X,A) \rightarrow (\connquadkarL(\ast))\otimes(X,A)$, which is an equivalence as observed in \cref{constr:assembly}.
\end{proof}

Our next objective is the study of the assembly map for the functor $\structure(-) \colon \spcpairtrans \rightarrow \spectra$. From the definitions, we directly compute
\begin{equation}
	 \structure(\ast)  \simeq \coconnquadL(\ast) \hspace{3mm} \text{and} \hspace{3mm} \structure(\ast,\ast) \simeq 0.
\end{equation}
Observe that $\structure(-)$ commutes with coproducts.
Hence, we may apply \cref{defn:approximating_functor_from_left_by_excisive} to $\structure(-) \colon \spcpairtrans \rightarrow \spectra$ to get a map which we call the \textit{structure spectrum assembly map} 
\begin{equation}
	\label{eq:structure_spectrum_assembly_map}
	(\coconnquadL(\ast)) \otimes (X,A) \rightarrow \structure(X,A).
\end{equation}

\begin{lem}
	\label{lem:when_structure_spectra_are_truncated_L_theory}
	The structure spectrum assembly map $(\coconnquadL(\ast)) \otimes (X,A) \rightarrow \structure(X,A)$ is an equivalence provided $(X,A)$ satisfies the following two conditions.
	\begin{enumerate}
		\item The spectrum $\structurekarper(X,A)$ vanishes.
		\item The map $\quadL(X,A) \rightarrow \quadkarL(X,A)$ is an equivalence.
	\end{enumerate}
\end{lem}

\begin{proof}
	Using the identification $\quadL(\ast) \simeq \quadkarL(\ast)$ we may draw the following commuting diagram.
	\begin{equation}
		\begin{tikzcd}
			& \quadL(X,A) \ar[r, "b"] \ar[d] & \quadkarL(X,A) \ar[d]\\
			\coconnquadL(\ast) \otimes (X,A) \ar[r, "a"] \ar[dr, dashed] & \structure(X,A) \ar[r] \ar[d] & \structurekarper(X,A) \ar[d]\\
			& \Sigma  (\connquadL (\ast)) \otimes (X,A) \ar[r, "c"] & \Sigma \quadL(\ast) \otimes (X,A)
		\end{tikzcd}
	\end{equation}
	We claim that the dashed arrow and the arrow labelled $c$ naturally refine to a fibre sequence in $\func^\codesc(\spcpairtrans,\spectra) \simeq \spectra^{\Delta^1}$. At $(\ast,\ast)$ the sequence evaluates to zero. At $(\ast,\emptyset)$, the morphism labelled $b$ becomes an equivalence, hence the lower square is cartesian. But further, $\structurekarper(\ast,\emptyset) \simeq 0$ and the morphism labelled $a$ becomes an equivalence by construction. This refines the lower composite to a fibre sequence, as claimed.
	
	Under the conditions of the lemma, the morphism labelled $b$ is an equivalence, hence the lower square cartesian, and further $\structurekarper(X,A) = 0$, so that $\structure(X,A)$ identifies with the fibre of the morphism labelled $c$, so indeed the morphism labelled $a$ is an equivalence.
\end{proof}

\subsection{Total surgery obstructions}

\label{subsec:recollect_tso}

In this section, we review Ranicki's \textit{total surgery obstruction}, and explain how it is used in this article about the Nielsen realisation problem.
Let $(X,\partial X)$ be an oriented Poincar\'e pair. We assume that it is \textit{finite}, i.e. that both $X$ and $\partial X$ have the homotopy type of finite CW complexes.
The total surgery obstruction, specialised to finite oriented Poincar\'e pairs $(X,\partial X)$ is an element
\begin{equation}
	\tso(X,\partial X) \in \pi_d \structure(X,\partial X).
\end{equation}
It vanishes when there is an equivalence of pairs $(X,\partial X) \simeq (M,\partial M)$ to a compact manifold with boundary. The converse is true if the dimension of $(X,\partial X)$ is at least 6. 

\begin{thm}[\cite{Ranicki92}, p. 210]
	\label{thm:Ranicki_tso}
	If $\dim(X,\partial X) \geq 6$, and $\tso(X,\partial X) = 0 \in \pi_d \structure(X,\partial X)$, then there exists a compact manifold pair $(M,\partial M)$ with a homotopy equivalence 
	\[ (M,\partial M) \xrightarrow{\simeq} (X,\partial X). \]
\end{thm}

\begin{rmk}[Compatibility with Poincar\'e bordism, cf. \cite{Ranicki92} Sec. 19 and Prop. 21.2.]
	\label{rmk:compatibility_with_pd_bordism}
	We briefly remark on the compatibility of the total surgery obstruction with Poincar\'e bordism.
	First, we describe a functor $\pdbord_*(-)$ from the category $\spcpairtrans$ to the category of graded abelian groups.
	\begin{enumerate}
		\item The group $\pdbord_d(X,A)$ is generated by maps $(Y,\partial Y) \rightarrow (X,A)$, where $(Y,\partial Y)$ is an oriented $d$-dimensional Poincar\'e pair. Such a generator is called a \textit{singular oriented Poincar\'e pair} in $(X,A)$. Two classes $(Y_i,\partial Y_i) \rightarrow (X,A)$ are declared equivalent provided there is a singular oriented Poincar\'e bordism of pairs between them.
		\item Given a map of pairs $f \colon (X,A) \rightarrow (X',A')$, postcomposing singular Poincar\'e pairs in $(X,A)$ with $f$ induces a map $\pdbord_*(X,A) \rightarrow \pdbord_*(X',A')$.
		\item Given a finite covering of pairs $p \colon (\overline{X},\overline{A}) \rightarrow (X,A)$, pulling back singular pairs in $(X,A)$ along $p$ induces a map $\pdbord_*(X,A) \rightarrow \pdbord_*(\overline{X},\overline{A})$.
	\end{enumerate}
	The spectrum $\structure(X,\partial X)$ is constructed so that there is a map of graded abelian groups 
	\begin{equation}
		\label{eq:pd_bordism_to_structure_spectrum}
		\pdbord_*(X,A) \rightarrow \pi_* \structure(X,A)
	\end{equation}
	natural in $\spcpairtrans$. If $(X,\partial X)$ is an oriented $d$-dimensional Poincar\'e pair, the identity defines a class in $\pdbord_d(X,\partial X)$, whose image under the map \cref{eq:pd_bordism_to_structure_spectrum} is the element $\tso(X,\partial X)$. 
\end{rmk}

The total surgery obstruction is well-behaved with respect to finite coverings: given a finite covering of finite Poincar\'e pairs $p \colon (Y,\partial Y) \rightarrow (X,\partial X)$, the induced map
\begin{equation}
	p^* \colon \pi_d \structure(X,\partial X) \rightarrow \pi_d \structure(Y,\partial Y)
\end{equation}
sends $\tso(X,\partial X)$ to $\tso(Y,\partial Y)$. This is a consequence of \cref{rmk:compatibility_with_pd_bordism}, see also \cite[Thm. 6.3.]{dl5}.

\section{Structure spectra of fixed-point complements}

Throughout this section, we place ourselves in the situation of \cref{sec:setup} and assume given a finite good decomposition of the $G$-space $\nielsspace = \pi \backslash E_\finite \Gamma$.
To fix notation, we write the corresponding pushout as
\begin{equation}
	\label{eq:good_decomposition}
	\begin{tikzcd}
		\partial C \ar[r] \ar[d] & C \ar[d] \\
		\nielsspace^{G} \ar[r] & \nielsspace
	\end{tikzcd}
\end{equation}
and write $(Q,\partial Q) \coloneqq G \backslash (C,\partial C)$ for the complement quotient. 
The ultimate goal of this section is to prove that under mild assumptions, the transfer map
\begin{equation}
	\pi_d \structure(Q,\partial Q) \rightarrow \pi_d\structure (C^{e},\partial C^{e})
\end{equation}
is injective. In light of \cref{subsec:recollect_tso}, this is useful to show that $\tso(Q,\partial Q)$ vanishes, reducing it to showing that $\tso(C^{e},\partial C^{e}) = 0$. To achieve this, we will actually fully compute the spectrum $\structure(Q,\partial Q)$ along the way. In fact, we show that the structure spectrum assembly map
\begin{equation}
	(\coconnquadL(\ast)) \otimes (Q,\partial Q) \xrightarrow{\simeq} \structure(Q,\partial Q)
\end{equation}
is an equivalence, by verifying the conditions of \cref{lem:when_structure_spectra_are_truncated_L_theory}.
The injectivity of the transfer map then proceeds by a careful comparison with the transfer in integral homology, see \cref{thm:transfer_injectivity}.

\subsection{Outline and first observations}

To show that the structure spectrum assembly map $(\coconnquadL(\ast)) \otimes (Q,\partial Q) \xrightarrow{\simeq} \structure(Q,\partial Q)$ is an equivalence, we have to verify two conditions from \cref{lem:when_structure_spectra_are_truncated_L_theory}.
\begin{enumerate}
	\item[(a)] The spectrum $\structurekarper(Q,\partial Q)$ vanishes.
	\item[(b)] The map $\quadL(Q,\partial Q) \rightarrow \quadkarL(Q,\partial Q)$ is an equivalence.
\end{enumerate}
Using \cref{prop:karoubi_invariant_L_theory_computation}, we may reformulate (b) as the statement that the map $\tateK(Q,\partial Q) \rightarrow \nctateK(Q,\partial Q)$ is an equivalence.  Both (a) and (b) are justified using the Farrell--Jones conjecture in algebraic $\ncKthy$- and $\Lthy$-theory. 

 Recall that we write $\orbifoldfundamentalgroup \coloneqq \pi_1((\nielsspace)_{hG})$. Write $\restrict_G^\orbifoldfundamentalgroup \colon \spc_G \rightarrow \spc_\orbifoldfundamentalgroup$ for the restriction functor along the quotient map $\orbifoldfundamentalgroup \rightarrow G$. Note that $\restrict_G^\orbifoldfundamentalgroup$ is right adjoint to the quotient functor $\pi \backslash (-) \colon \spc_\orbifoldfundamentalgroup \rightarrow \spc_G$, so we in particular get an adjunction unit
\begin{equation}
	q \colon E_\finite \orbifoldfundamentalgroup \rightarrow \restrict_G^\orbifoldfundamentalgroup (\pi \backslash E_\finite \orbifoldfundamentalgroup) \simeq \restrict_G^\orbifoldfundamentalgroup \nielsspace.
\end{equation}
Base change along $q$ defines a functor $q^\ast \colon (\spc_{\orbifoldfundamentalgroup})_{/\restrict_G^\orbifoldfundamentalgroup \nielsspace} \rightarrow (\spc_\orbifoldfundamentalgroup)_{/E_\finite \orbifoldfundamentalgroup}$.
Consider the composite
\begin{equation}
	(\widehat{-}) \colon
	\spc_{G} \xrightarrow{\restrict_G^\orbifoldfundamentalgroup} (\spc_{\orbifoldfundamentalgroup})_{/\restrict_G^\orbifoldfundamentalgroup \nielsspace} \xrightarrow{q^{\ast}} (\spc_\orbifoldfundamentalgroup)_{/E_\finite \orbifoldfundamentalgroup}.
\end{equation}
Since both $\restrict_G^\orbifoldfundamentalgroup$ and $q^\ast$ are colimit preserving functors, also $(\widehat{-})$ preserves colimits.
Applying $(\widehat{-})$ to the square \cref{eq:good_decomposition} we obtain a pushout of $\orbifoldfundamentalgroup$-spaces
\begin{equation}
	\label{eq:good_decomposition_on_covers}
	\begin{tikzcd}
		\widehat{\partial C}  \ar[r] \ar[d] & \widehat{C} \ar[d] \\
		\widehat{\nielsspace^G} \ar[r] & E_\finite \orbifoldfundamentalgroup.
	\end{tikzcd}
\end{equation}
In the next section, we use \textit{Davis--L\"uck assembly} to reduce both (a) and (b) to computations with certain homology theories on the category of $\orbifoldfundamentalgroup$-spaces. For these computations, it is essential that we understand the lower horizontal map in \cref{eq:good_decomposition_on_covers}.
For this, we make the following definition.
\begin{defn}
	\label{def:full_subgroup}
	A subgroup $F \subset \orbifoldfundamentalgroup$ is \textit{full} if the composite $F \subset \orbifoldfundamentalgroup \rightarrow G$ is an isomorphism. We write $\maximal$ for a set of representatives of conjugacy classes of full finite subgroups of $\orbifoldfundamentalgroup$.
\end{defn}

\begin{lem}
	If $\nielsspace$ admits a good decomposition, then 
	\begin{equation}
		\widehat{\nielsspace^G} \simeq \coprod_{F \in \maximal} \induct_{N_\orbifoldfundamentalgroup F}^\orbifoldfundamentalgroup E_\finite N_\orbifoldfundamentalgroup F  \hspace{3mm} \text{and} \hspace{3mm} \nielsspace^G \simeq \coprod_{F \in \maximal} BW_\orbifoldfundamentalgroup F.
	\end{equation}
\end{lem}

\begin{proof}
	Since both $\widetilde{\partial C}$ and $\widehat{C}$ are built from free orbits, we see that $\widehat{\nielsspace^G} \simeq (E_\finite \orbifoldfundamentalgroup)^{>1}$. Note that since the $G$-action on $\nielsspace^G$ is trivial, $\widehat{\nielsspace^G}$ is built from cells with full isotropy groups. So \cref{appendixlemma:characterisations_of_property_M} and \cref{appendixlemma:computing_fixed_points} give the assertions.
\end{proof}

\subsection{Davis--L\"uck assembly and the Farrell--Jones conjecture}
\label{subsec:Davis_Lueck_assembly}

We want to prove two results about the complement quotient $(Q,\partial Q)$.
\begin{enumerate}
	\item The map $\tateK(Q,\partial Q) \rightarrow \nctateK(Q,\partial Q)$ is an equivalence.
	\item The periodic Karoubi structure spectrum $\structurekarper(Q,\partial Q)$ vanishes.
\end{enumerate}
To prove these assertions, we will employ the technique of \textit{Davis--L\"uck assembly}. This is a convenient framework to compute with Farrell--Jones assembly maps.

Let $\Gamma$ be a discrete group. Recall that the category of $\Gamma$-spaces $\spc_\Gamma = \PSh(\orbit(\Gamma))$ is freely generated under colimits by the orbit category $\orbit(\Gamma)$. That means that a colimit preserving functor $\spc_\Gamma \rightarrow \spectra$ is determined by its restriction to the orbit category $\orbit(\Gamma) \subset \spc_\Gamma$, and the inverse of this procedure is given by left Kan extension.

\begin{constr}[Davis--L\"uck equivariant homology theories, \cite{dl98}]
	\label{cons:equivariant_homology_theories}
	Consider a functor
	\begin{equation}
		\generichlgy \colon \Gpd \rightarrow \spectra
	\end{equation}
	from the category of (1--)groupoids to the category of spectra, which we will refer to as a \textit{spectral equivariant homology theory}. Given a discrete group $\Gamma$, we obtain a functor $\generichlgy^\Gamma \colon \spc_\Gamma \rightarrow \spectra$ by left Kan extending the composite
	\begin{equation}
		\orbit(\Gamma) \xrightarrow{(-)_{h\Gamma}} \Gpd \xrightarrow{\generichlgy} \spectra
	\end{equation}
	to $\spc_\Gamma$. We will refer to $\generichlgy^\Gamma$ as the \textit{spectral $\Gamma$-homology theory associated to $\generichlgy$}. Here, we have used that for a transitive $\Gamma$-set $S$, the homotopy orbits $S_{h\Gamma}$ are a groupoid. In fact, $(\Gamma/H)_{h\Gamma} \simeq BH$. Given a group homomorphism $\alpha \colon \Gamma \rightarrow \Gamma'$ we get an induced functor
	\begin{equation}
		\induct_\alpha \colon \orbit(\Gamma) \rightarrow \orbit(\Gamma').
	\end{equation}
	The left diagram below evidently commutes, providing a lax Beck-Chevalley square in the right diagram below.
	\begin{equation}
		\begin{tikzcd}
			\spc_\Gamma \ar[d] && \spc_{\Gamma'} \ar[d]  \ar[ll, "\restrict_\alpha"']& \spc_\Gamma \ar[drr, phantom, "\implies"]  \ar[d] \ar[rr, "\induct_\alpha"] && \spc_{\Gamma'} \ar[d]\\
			\spc^{B\Gamma}  && \spc^{B\Gamma'}  \ar[ll, "\restrict_\alpha"']& \spc^{B\Gamma} \ar[rr, "\induct_\alpha"] \ar[dr, "(-)_{h\Gamma}"'] && \spc^{B\Gamma'} \ar[dl, "(-)_{h\Gamma'}"]\\
			& \spc \ar[ur] \ar[ul] &&& \spc &
		\end{tikzcd}
	\end{equation}
	Restricting along $\orbit(\Gamma) \subset \spc_\Gamma$, we obtain a comparison map $w_\alpha \colon (-)_{h\Gamma} \rightarrow (\induct_\alpha(-))_{h\Gamma'}$. Left Kan extending to $\spc_\Gamma$ we get a comparison map
	\begin{equation}
		\generichlgy(\alpha) \colon \generichlgy^\Gamma \rightarrow \generichlgy^{\Gamma'} \circ \induct_\alpha.
	\end{equation}
	On the orbit $\Gamma/H$, the map $w_\alpha$ is the evident map $BH \rightarrow B\alpha(H)$, where $\alpha(H)$ is the image subgroup of $H$ under $\alpha$ in $\Gamma'$. Therefore, $w_\alpha$ is an equivalence on those orbits $\Gamma/H$ where $H$ intersects the kernel of $\alpha$ trivially. Hence, $\generichlgy(\alpha)$ is an equivalence on the category of all $\Gamma$-spaces generated under colimits by the orbits $\{ \Gamma/H \mid H \cap \ker(\alpha) = \{e\}\}$.
\end{constr}

\begin{example}
	\label{ex:spectral_equivariant_homology_theories_from_localising_invariants}
	For $E \in \{ \Kthy, \ncKthy, \tateK,\nctateK,\Lthy,\karLthy \}$ we have the functor
	\begin{equation}
		E \colon \spc \rightarrow \spectra, \hspace{3mm} X \mapsto E( (\module{\bbZ}^\finitemodule,\quadratic)_X )
	\end{equation}
	We may restrict $E$ along the inclusion functor $\Gpd \subset \spc$ to obtain a spectral equivariant homology theory as in \cref{cons:equivariant_homology_theories}. We write 
	\begin{equation}
		H^\Gamma(- ; E) \colon \spc_\Gamma \rightarrow \spectra
	\end{equation}
	for the associated spectral $\Gamma$-homology theory. Note for example, that $H^\Gamma(\Gamma/H;\Lthy) = \quadL(BH)$.
\end{example}

A use of Davis--L\"uck assembly is that certain assembly maps can be studied in terms of equivariant homology theories. This is in particular useful in the context of the Farrell--Jones conjecture. Recall that a \textit{family of subgroups} of a discrete group $\Gamma$ is a collection of subgroups closed under passing to subgroups and under conjugation. Let $\family$ be a family of subgroups of $\Gamma$. The \textit{universal $\Gamma$-space for $\family$} is the $\Gamma$-space $E_\family \Gamma$ with
\begin{equation}
	(E_\family \Gamma)^H \simeq \begin{cases}
		* \colon & \text{if $H\in \family$};\\
		\emptyset \colon & \text{otherwise}.
	\end{cases}
\end{equation}
Typical families we consider are the families $\finite$ and $\vcyc$ of finite and virtually cyclic subgroups. Given a family $\family$ we write $\orbit_\family(\Gamma)$ for the subcategory of $\orbit(\Gamma)$ spanned by those orbits having isotropy contained in $\family$.

\begin{lem}[Davis--L\"uck assembly]
	\label{lem:davis_lueck_assembly}
	Let $\Gamma$ be a discrete group, and $\generichlgy \colon \Gpd \rightarrow \spectra$ a functor. The two maps
	\begin{equation}
		\colim_{S \in \orbit_\family(\Gamma)} \generichlgy(S_{h\Gamma}) \rightarrow \generichlgy(B\Gamma) \hspace{5mm} \text{and} \hspace{5mm} \generichlgy^\Gamma(E_\family \Gamma) \rightarrow\generichlgy^\Gamma(*)
	\end{equation}
	are equivalent.
\end{lem}

\begin{proof}
	See \cite[Sec. 2]{BunkeKasprowskiWinges}.
\end{proof}

Next, we recall the Farrell--Jones conjecture. After that, we briefly recall a large class of groups where it is known to hold. In the rest of the article, we often use it as an input for computations.

\begin{conj}[The Farrell--Jones conjecture, integer coefficients]
	\label{conj:farrell_jones}
	Let $\Gamma$ be a discrete group, and $\vcyc$ the family of virtually cyclic subgroups of $\Gamma$. Then the assembly maps
	\begin{equation}
		H^\Gamma(E_\vcyc \Gamma; \ncKthy) \rightarrow H^\Gamma(\ast;\ncKthy) \hspace{3mm} \text{and} \hspace{3mm} H^\Gamma(E_\vcyc \Gamma;\karLthy) \rightarrow H^\Gamma(\ast;\karLthy)
	\end{equation}
	are equivalences. 
\end{conj}

\begin{recollect}
	\label{recollect:Farrell_Jones_groups}
	Recall the notion of a Farrell--Jones group from \cite{lueckIC}. For every Farrell--Jones group $\Gamma$, \cref{conj:farrell_jones} is known to be true. Farrell--Jones groups are closed under passing to subgroups and passing to overgroups with finite index. Word-hyperbolic groups are Farrell--Jones groups, see \cite{BLR}. See \cite[Sec. 16.2.]{lueckIC} for a comprehensive survey of the current status and \cite{BunkeKasprowskiWinges} for a higher categorical approach in the case of $K$-theory. The $L$-theoretic analog for Poincaré categories with a group action, is work in progress by Christoph Winges \cite{ChristophFarrelJones}.
\end{recollect}

We need the following companion results on the Farrell--Jones conjecture which we adapt to our modernised setup.

\begin{rmk}[The transitivity principle]
	\label{rmk:transitivity_principle}
	Let $\generichlgy \colon \Gpd \rightarrow \spectra$ be a spectral equivariant homology theory, and $\Gamma$ a discrete group. Suppose given two families of subgroups $\family \subset \family'$ of $\Gamma$.
	For each $H \in \family'$, write $\family \cap H$ for the family of subgroups of $H$ which lie in $\family$. Assume that for each $H \in \family'$ the map
	\begin{equation}
		\label{eq:transitivity_principle_condition}
		\generichlgy^H(E_{\family \cap H} H) \rightarrow \generichlgy^H(\ast)
	\end{equation}
	is an equivalence. 
	Then, if $Z \in \spc_\Gamma$ has isotropy in the family $\family'$, also the map
	\begin{equation}
		\generichlgy^\Gamma(Z \times E_{\family} \Gamma) \rightarrow \generichlgy^\Gamma(Z)
	\end{equation}
	is an equivalence. So in particular for $Z=E_\family' \Gamma$ the map $\generichlgy^\Gamma(E_{\family}\Gamma) \rightarrow \generichlgy^\Gamma(E_{\family'}\Gamma)$ is an equivalence.
	The proof in \cite[Sec. 15.5]{lueckIC} carries over verbatim. 
%	We may reduce to the case of an orbit $Z=\Gamma/H$ for $H \in \family'$, and use the induction homomorphism to indentify the map $\generichlgy^\Gamma(\Gamma/H \times E_\family \Gamma) \rightarrow  \generichlgy(\Gamma/H)$ with the map \cref{eq:transitivity_principle_condition}. 
\end{rmk}

\begin{lem}
	\label{lem:Farrell_Jones_for_larger_families}
	If $\Gamma$ is a discrete group, and all subgroups of $\Gamma$ satisfy the Farrell--Jones conjecture \cref{conj:farrell_jones}, then if $\family$ is any family containing all virtually cyclic subgroups, the assembly maps
	\begin{equation}
		H^\Gamma(E_\family \Gamma; \ncKthy) \rightarrow H^\Gamma(\ast;\ncKthy) \hspace{3mm} \text{and} \hspace{3mm} H^\Gamma(E_\family \Gamma;\karLthy) \rightarrow H^\Gamma(\ast;\karLthy)
	\end{equation}
	are equivalences.
\end{lem}

\begin{proof}
	If $H \subset \Gamma$ is a subgroup, then $\vcyc \cap H$ is the family of virtually cyclic subgroups of $H$. So the transitivity principle \cref{rmk:transitivity_principle} implies that the maps
	\begin{equation}
		H^\Gamma(E_\vcyc \Gamma; \ncKthy) \rightarrow H^\Gamma(E_\family \Gamma;\ncKthy) \hspace{3mm} \text{and} \hspace{3mm} H^\Gamma(E_\vcyc \Gamma;\karLthy) \rightarrow H^\Gamma(E_\family \Gamma ;\karLthy)
	\end{equation}
	are equivalences.  Combining this with the Farrell--Jones conjecture for $\Gamma$ gives the result.
\end{proof}

\begin{recollect}[Virtually cyclic subgroups of type I and II, \cite{lueckIC}, Lem. 13.42.]
	For the next lemma, we recall that virtually cyclic groups come in two varieties. A virtually cyclic group $V$ is of \textit{type I} if it is finite or admits a surjection $p \colon V \rightarrow \bbZ$ to the group of integers. If $V$ is not of type I, we say it is of type II. Every type II virtually cyclic subgroup has elements of even order. The $\Lthy$-theory of group rings on type II virtually cyclic subgroups is quite complicated, see \cite{ConnollyDavis04}. For type I virtually cyclic subgroups, the situation is considerably simpler.
\end{recollect}

\begin{lem}
	\label{lem:relative_assembly_map_and_l_theory}
	If $\Gamma$ is a discrete group which contains no virtually cyclic subgroups of type II (e.g. $\Gamma$ has no elements of even order), then the map
	\begin{equation}
		\label{eq:relative_assembly_map}
		H^\Gamma(E_\finite \Gamma;\karLthy) \rightarrow H^\Gamma(E_\vcyc \Gamma;\karLthy)
	\end{equation}
	is an equivalence.
\end{lem}

\begin{proof}
	By the transitivity principle \cref{rmk:transitivity_principle} we may reduce to the case where $\Gamma$ itself is a virtually cyclic subgroup of type I. 
	If $\Gamma$ is finite, the statement is clear. If $\Gamma$ is infinite, it has a unique maximal finite subgroup $F \subset \Gamma$, and $\Gamma/F$ is infinite cyclic. Since $F$ is a unique maximal subgroup, 
	the full subcategory $\{\Gamma/F\}$ of $\orbit_\finite(\Gamma)$ spanned by $\Gamma/F$ is final. 
	
	The subcategory spanned by $\Gamma/G$ is equivalent to $BW_\Gamma F \simeq B\bbZ$, and the map $B\Gamma \simeq BN_\Gamma F \rightarrow BW_\Gamma F \simeq B\bbZ$ corresponds to a $\bbZ$-action on the space $BF$. By Davis--L\"uck assembly \cref{lem:davis_lueck_assembly} and finality of $\{\Gamma/F\} \subset \orbit_\finite(\Gamma)$, the map \cref{eq:relative_assembly_map} identifies with the assembly map
	\begin{equation}
		\colim_{B\bbZ} \karLthy( (\module{\bbZ}^\finitemodule,\quadratic)_{BF} ) \rightarrow \karLthy( (\module{\bbZ}^\finitemodule,\quadratic)_{B\Gamma} ).
	\end{equation}
	But this map is an equivalence by the twisted Shaneson splitting, see \cite[Sec. 4]{levin2025modelassemblymapbordisminvariant}.
\end{proof}

\subsection{$\Lthy$-theoretic computations}

In this subsection, we use Davis--L\"uck's assembly method to prove the following result. Recall the notations $\orbifoldfundamentalgroup$ and $\nielsspace$ from \cref{sec:setup}.

\begin{thm}
	\label{thm:periodic_karoubi_structure_spectrum_vanishing}
	Assume that $\orbifoldfundamentalgroup$ is a Farrell--Jones group which contains no virtually cyclic subgroups of type II.
	For any finite good decomposition of $\nielsspace$ with complement quotient $(Q,\partial Q)$, we have
	\begin{equation}
		\structurekarper(Q,\partial Q) \simeq 0.
	\end{equation}
\end{thm}

\begin{proof}
	Using the commutative diagram \cref{eq:good_decomposition_on_covers} we build the following diagram of spectra.
	\begin{equation}
		\label{eq:structure_spectrum_computation_diagram}
		\begin{tikzcd}
			H^\orbifoldfundamentalgroup(\widehat{\partial C};\karLthy) \ar[r] \ar[d] & \widehat{C} \ar[d] \\
			H^\orbifoldfundamentalgroup(\widehat{\nielsspace^G} \simeq
			\coprod_{F \in \maximal} \induct_{N_\orbifoldfundamentalgroup F}^\orbifoldfundamentalgroup E_\finite N_\orbifoldfundamentalgroup F;\karLthy)  \ar[r] \ar[d] & H^\orbifoldfundamentalgroup(E_\finite \orbifoldfundamentalgroup;\karLthy)  \ar[d]\\
			H^\orbifoldfundamentalgroup(\coprod_{F \in \maximal} \induct_{N_\orbifoldfundamentalgroup F}^\orbifoldfundamentalgroup E_\vcyc N_\orbifoldfundamentalgroup F;\karLthy)  \ar[r] \ar[d] & H^\orbifoldfundamentalgroup(E_\vcyc \orbifoldfundamentalgroup;\karLthy)  \ar[d]\\
			H^\orbifoldfundamentalgroup(\coprod_{F \in \maximal} \orbifoldfundamentalgroup/ N_\orbifoldfundamentalgroup F;\karLthy)  \ar[r] & H^\orbifoldfundamentalgroup(\ast;\karLthy) 
		\end{tikzcd}
	\end{equation}
	Unraveling identifies $\structurekarper(Q,\partial Q)$ with the total cofibre of the outermost square of \cref{eq:structure_spectrum_computation_diagram}. We argue that all three squares are pushout squares. This implies the result, since then the outer square is a pushout as well, and the total cofibre of a pushout square is zero.
	
	The first square is a pushout square, since it is constructed by applying  $H^\orbifoldfundamentalgroup(-;\karLthy)$ to a pushout of $\orbifoldfundamentalgroup$-spaces. In the second square, both vertical maps are equivalences by \cref{lem:relative_assembly_map_and_l_theory}, using that $\orbifoldfundamentalgroup$ has no virtually cyclic subgroups of type II. The vertical maps in the third square are equivalences as a consequence of the Farrell--Jones conjecture for $\karLthy$.
\end{proof}

\subsection{$\ncKthy$-theory}

Our next objective is to show that the map $\tateK(Q,\partial Q) \rightarrow \nctateK(Q,\partial Q)$ is an equivalence. 
We start with a preliminary lemma, that reduces to a statement about $\ncKthy$-without taking Tate fixed points. This is practical for considerations with the Farrell--Jones conjecture: the Farrell--Jones conjecture involves an infinite colimit, and the Tate construction does not commute with all colimits. 
\begin{lem}
	\label{lem:tateK_criterion}
	Let $(Y,A)$ be a pair of spaces. Suppose the following assumptions are satisfied.
	\begin{enumerate}
		\item The map $\ncKthy_0(A) /\Kthy_0(A) \rightarrow \ncKthy_0(Y)/\Kthy_0(Y)$ is an isomorphism.
		\item For each $i \leq -1$, the map $\ncKthy_i(A) \rightarrow \ncKthy_i(Y)$ is an isomorphism.
	\end{enumerate}
	Then the map $\tateK(Y,A) \rightarrow \nctateK(Y,A)$ is an isomorphism. 
\end{lem}
\begin{proof}
	Recall the definition of $\Kthy(Y,A)$ as $\Kthy((\module{\bbZ}^\finitemodule)_{(Y,A)})$, as constructed by \cref{eq:relative_metabolic_construction}. Note that the stable category underlying $\Met(\module{\bbZ}^\finitemodule)_Y$ is the arrow category of $(\module{\bbZ}^\finitemodule)_Y$, the tensor of $\module{\bbZ}^\finitemodule$ over $Y$. Hence $\Kthy(\Met(\module{\bbZ}^\finitemodule)_Y) \simeq \Kthy(Y) \oplus \Kthy(Y)$. This leads to the observation, that without regard for the $C_2$-action $\Kthy(Y,A) \simeq \Kthy(Y) \oplus \Kthy(A)$. Similarly $\ncKthy(Y,A) \simeq \ncKthy(Y) \oplus \ncKthy(A)$.
	
	With these preliminary observations in mind, consider the following commuting diagram, where $\alpha$ and $\beta$ are induced by the inclusion of pairs $(A,A) \subset (Y,A)$.
	\begin{equation}
		\begin{tikzcd}
			\Kthy(A,A) \ar[r, "\alpha"] \ar[d] & \Kthy(Y,A) \ar[d] \\
			\ncKthy(A,A) \ar[r, "\beta"]  & \ncKthy(Y,A)
		\end{tikzcd}
	\end{equation}
	Since $\tateK(A,A) \simeq \nctateK(A,A) \simeq 0$ by bordism invariance, we observe that the map $\tateK(Y,A) \rightarrow \nctateK(Y,A)$ is an equivalence if and only if the map $\cofib(\alpha)^{tC_2} \rightarrow \cofib(\beta)^{tC_2}$ is an equivalence. This is true, for example, if before taking Tate fixed points, the map $\cofib(\alpha) \rightarrow \cofib(\beta)$ is an equivalence. 
	
	Our preliminary observations identify this map exactly with the map $\Kthy(Y) / \Kthy(A) \rightarrow \ncKthy(Y)/\ncKthy(A)$. This map is an equivalence if and only if the map $\ncKthy(A) /\Kthy(A) \rightarrow \ncKthy(Y)/\Kthy(Y)$ is an equivalence. We claim that this reduces to the two claims in the statement of the lemma. For any space $Y$, we have that $\pi_\ast \ncKthy(Y)/\Kthy(Y)$ vanishes for $\ast \geq 2$. For $\ast = 1,0$ we consider the exact sequence
	\begin{equation}
		\Kthy_1(Y) \xrightarrow{i_1} \ncKthy_1(Y) \rightarrow \pi_1 \ncKthy(Y)/\Kthy(Y) \rightarrow \Kthy_0(Y) \xrightarrow{i_0} \ncKthy_0(Y) \rightarrow \pi_0 \ncKthy(Y)/\Kthy(Y) \rightarrow \Kthy_{-1}(Y).
	\end{equation}
	First, the map $i_1$ is an isomorphism and $i_0$ is an injective map. This shows that $\pi_1 \ncKthy(Y)/\Kthy(Y) = 0$. Second, $\Kthy_{-1}(Y) = 0$, so the sequence induces an isomorphism $\ncKthy_0(Y)/\Kthy_0(Y) \simeq \pi_0 \ncKthy(Y)/\Kthy(Y)$. Lastly, we observe that for $\ast \leq -1$ we have that the map $ \pi_\ast \ncKthy(Y) \rightarrow \pi_\ast \ncKthy(Y)/\Kthy(Y)$ is an isomorphism.
	Hence, the map $\pi_\ast \ncKthy(A)/\Kthy(A) \rightarrow \pi_\ast \ncKthy(Y)/\Kthy(Y)$ is always an isomorphism for $\ast \geq 1$, under the first condition for $\ast = 0$ and under the second condition for $\ast \leq -1$.
\end{proof}

\begin{obs}
	\label{obs:tateK_criterion_only_depends_on_1_type}
	Note that the criterion derived in \cref{lem:tateK_criterion} only depends on the $1$-type of $(Y,A)$. For relevance to us, if $(Y,A) \rightarrow (Y',A')$ is a $\pi$--$\pi$-equivalence of pairs, the criterion is satisfied for $(Y,A)$ if and only if it is satisfied for $(Y',A')$.
\end{obs}

\begin{lem}
	\label{lem:Khty_and_torsionfree_stuff}
	Let $\Gamma$ be a Farrell--Jones group.
	Suppose that $X$ is a $\Gamma$-space with torsion-free isotropy groups. Then the map
	\begin{equation}
		H^\Gamma(X;\Kthy) \rightarrow H^\Gamma(X;\ncKthy)
	\end{equation}
	is an isomorphism.
\end{lem}

\begin{proof}
	Both sides commute with colimits of spaces with torsion-free isotropy groups, so the statement reduces to the case $X = \Gamma/V$, where $V$ is a torsionfree subgroup. In this case, it unravels to the map
	\begin{equation}
		\Kthy(\module{\bbZ V}^\finitemodule) \rightarrow \ncKthy(\module{\bbZ V}^\finitemodule)
	\end{equation}
	being an isomorphism. In other words, we have to show that $\ncKthy(\module{\bbZ V}^\finitemodule)$ is connective and that $\ncKthy_0(\module{\bbZ V}^\finitemodule)$ is cyclic, generated by a free module of rank one. Since $\bbZ$ is a regular ring, this is a consequence of the Farrell--Jones conjecture for $V$ \cite[Conj. 2.60 and 4.18]{lueckIC}.
\end{proof}

\begin{lem}
	\label{lem:insensitivity_of_connective_k_theory_to_isotropy_in_nonpositive_degrees}
	Let $\Gamma$ be any discrete group, and let $f \colon X \rightarrow Y$ be any map of $\Gamma$-spaces. If $f^{e} \colon X^{e}_{h\Gamma} \rightarrow Y^{e}_{h\Gamma}$ induces a bijection on path components, the induced map
	\begin{equation}
		\pi_\ast H^\Gamma(X;\Kthy) \rightarrow \pi_\ast H^\Gamma(Y;\Kthy)
	\end{equation}
	is an isomorphism for $\ast \leq 0$.
\end{lem}

\begin{proof}
	For $\ast \leq -1$, both sides are zero, since $H^\Gamma(-;\Kthy)$ has connective coefficients, and connective spectra are closed under colimits. For $\ast = 0$, observe that we can reduce to the map $Y \times E\Gamma \rightarrow Y$, since $\pi_0 H^\Gamma(Y \times E \Gamma;\Kthy) \simeq \pi_0 \Kthy(\ast) \otimes Y^{e}_{h\Gamma}$ only depends on the path components of $Y^{e}_{h\Gamma}$. For this case, observe that $\pi_0$ commutes with colimits of connective spectra, hence we may reduce further to the case of an orbit $\Gamma/H$. But the map $\pi_0 H^\Gamma(\Gamma/H \times E\Gamma;\Kthy) \rightarrow H^\Gamma(\Gamma/H;\Kthy)$ identifies with the map $\pi_0 \Kthy(\module{\bbZ}^\finitemodule) \otimes BH_+ \rightarrow \pi_0 \Kthy(\module{\bbZ H}^\finitemodule)$, which is an isomorphism: both sides are isomorphic to $\bbZ$, detected by the rank of a free $\bbZ$- resp. $\bbZ H$-module.
\end{proof}

After these general considerations, we again specialise to the study of complement quotients $(	Q,\partial Q)$ and the orbifold fundamental group $\orbifoldfundamentalgroup$.
To this end, we introduce a few families of subgroups.
Let $\torsionfree$ be the family of torsionfree subgroups of the orbifold fundamental group $\orbifoldfundamentalgroup$. Recall the notion of a full subgroup of $\orbifoldfundamentalgroup$ from \cref{def:full_subgroup}.
We define a family of subgroups $\normfam$ by
\begin{equation}
	\normfam \coloneqq \{ H \subset \orbifoldfundamentalgroup \mid \text{$H$ is torsionfree or normalises a full subgroup} \}.
\end{equation}
\begin{obs}
	\label{obs:normfam_and_vcyc}
	Note that $\normfam$ is indeed a family, and $\torsionfree \subset \normfam$. 
	By the second point of \cref{appendixlemma:normalisers_and_virtually_cyclic_subgroups},  every type I virtually cyclic subgroup of $\orbifoldfundamentalgroup$ is contained in $\normfam$. 
\end{obs}

\begin{lem}
	\label{lem:pushout_for_normaliser_family}
	There is a pushout of $\orbifoldfundamentalgroup$-spaces as follows.
	\begin{equation}
		\label{eq:pushout_for_normaliser_family}
		\begin{tikzcd}
			\coprod_{F \in \maximal} (\orbifoldfundamentalgroup/N_\orbifoldfundamentalgroup F) \times E_\torsionfree \orbifoldfundamentalgroup \ar[r] \ar[d] & E_\torsionfree \orbifoldfundamentalgroup \ar[d]\\
			\coprod_{F \in \maximal} \orbifoldfundamentalgroup/ N_\orbifoldfundamentalgroup F \ar[r] & E_{\normfam} \orbifoldfundamentalgroup
		\end{tikzcd}
	\end{equation}
\end{lem}

\begin{proof}
	We may compute that the square is a pushout after passing to $H$-fixed points for all subgroups $H \subset \orbifoldfundamentalgroup$. If $H \in \torsionfree$ or $H \notin \normfam$, this is clear. If $H \in \normfam$ but $H \notin \torsionfree$, then after applying $(-)^{H}$ the upper horizontal map in \cref{eq:pushout_for_normaliser_family} is an equivalence, and the lower right corner is contractible. Hence, we have to show that the lower left corner in \cref{eq:pushout_for_normaliser_family} has exactly one $H$-fixed point.  Suppose that $HgN_\orbifoldfundamentalgroup F = gN_\orbifoldfundamentalgroup F$. Then $H \subset gN_\orbifoldfundamentalgroup F g^{-1} = N_\orbifoldfundamentalgroup gFg^{-1}$. In other words, the set of $H$-fixed points in question is in bijection with the set of full subgroups $\widetilde{F}$ with $H \subset N_\orbifoldfundamentalgroup \widetilde{F}$.
	
	Since $H \in \normfam$ is not torsionfree, this set is non-empty. To show that it has at most one point, observe that $H$ contains a finite subgroup $H_\finite$.  If $H \subset N_\orbifoldfundamentalgroup \widetilde{F}$, then $H_\finite \subset \widetilde{F}$ by \cref{appendixlem:group_theory}. But $H_\finite$ is contained in exactly one full subgroup by \cref{appendixlemma:characterisations_of_property_M}.
\end{proof}

\begin{thm}
	\label{thm:tateK_computation}
	Suppose $\nielsspace$ admits a good decomposition \cref{eq:good_decomposition} and that $\orbifoldfundamentalgroup$ is a Farrell--Jones group without virtually cyclic subgroups of type II. Then the comparison map induces an equivalence
	\begin{equation}
		\tateK(Q,\partial Q) \xrightarrow{\simeq} \nctateK(Q,\partial Q).
	\end{equation}
\end{thm}

\begin{proof}
	For this proof, we write $S$ for the $\orbifoldfundamentalgroup$-set $\coprod_{F \in \maximal} \orbifoldfundamentalgroup/N_\orbifoldfundamentalgroup F$. Observe that there is a $\pi$--$\pi$-equivalence $(Q,\partial Q) \rightarrow (\orbifoldfundamentalgroup/\orbifoldfundamentalgroup_{h\orbifoldfundamentalgroup},S_{h\orbifoldfundamentalgroup})$ by \cref{appendixlemma:computing_fixed_points}. Consider the following square.
	\begin{equation}
		\label{diag:Kthy_computation_square}
		\begin{tikzcd}
			H^\orbifoldfundamentalgroup_\ast(S; \Kthy) \ar[r] \ar[d] & H^\orbifoldfundamentalgroup_\ast(\orbifoldfundamentalgroup/\orbifoldfundamentalgroup;\Kthy) \ar[d] \\
			H^\orbifoldfundamentalgroup_\ast(S;\ncKthy) \ar[r] & H^\orbifoldfundamentalgroup_\ast(\orbifoldfundamentalgroup/\orbifoldfundamentalgroup;\ncKthy)
		\end{tikzcd}
	\end{equation}
	By \cref{lem:tateK_criterion} and \cref{obs:tateK_criterion_only_depends_on_1_type} it suffices to show that for $\ast \leq 0$, the map induced on vertical cokernels in \cref{diag:Kthy_computation_square} is an isomorphism.  We can precompose these vertical maps by the maps induced by the projections $E_\torsionfree \orbifoldfundamentalgroup \rightarrow \orbifoldfundamentalgroup/\orbifoldfundamentalgroup$ and $E_\torsionfree \orbifoldfundamentalgroup \times S \rightarrow S$. For $T = S$ resp. $T= \orbifoldfundamentalgroup/\orbifoldfundamentalgroup$, the composite
	\begin{equation}
		\label{eq:trick_with_torsionfree_family}
		H_\ast^\orbifoldfundamentalgroup(T \times E_\torsionfree \orbifoldfundamentalgroup; \Kthy) \rightarrow H^\orbifoldfundamentalgroup_\ast(T ;\Kthy) \rightarrow H^\orbifoldfundamentalgroup_\ast(T;\ncKthy)
	\end{equation}
	has the same image for $\ast \leq 0$ as left resp. right vertical map in \cref{diag:Kthy_computation_square} by \cref{lem:insensitivity_of_connective_k_theory_to_isotropy_in_nonpositive_degrees}. With this observation in mind, we factor the map \cref{eq:trick_with_torsionfree_family} to draw the following diagram.
	\begin{equation}
		\label{diag:bigger_kthy_computation_diagram}
		\begin{tikzcd}
			H^\orbifoldfundamentalgroup_\ast(S \times E_\torsionfree \orbifoldfundamentalgroup; \Kthy) \ar[r] \ar[d] & H^\orbifoldfundamentalgroup_\ast(E_\torsionfree \orbifoldfundamentalgroup; \Kthy) \ar[d] \\
			H^\orbifoldfundamentalgroup_\ast(S \times E_\torsionfree \orbifoldfundamentalgroup; \ncKthy) \ar[r] \ar[d] & H^\orbifoldfundamentalgroup_\ast(E_\torsionfree \orbifoldfundamentalgroup; \ncKthy) \ar[d] \\
			H^\orbifoldfundamentalgroup_\ast(S; \ncKthy) \ar[r] \ar[dr] & H^\orbifoldfundamentalgroup_\ast(E_\normfam \orbifoldfundamentalgroup; \ncKthy) \ar[d, "q"] \\
			& H^\orbifoldfundamentalgroup_\ast(\orbifoldfundamentalgroup/\orbifoldfundamentalgroup;\ncKthy)
		\end{tikzcd}
	\end{equation}
	We argued that for $\ast \leq 0$ the image of the vertical composites in \cref{diag:bigger_kthy_computation_diagram} agrees with the image of the vertical maps in \cref{diag:Kthy_computation_square}. By \cref{lem:Khty_and_torsionfree_stuff}, the vertical maps in the upper square are isomorphisms. Using the assumption that $\orbifoldfundamentalgroup$ contains no type II virtually cyclic subgroups, \cref{obs:normfam_and_vcyc} shows that we have an inclusion of families $\vcyc \subset \normfam$.
	As $\orbifoldfundamentalgroup$ is Farrell--Jones group, \cref{lem:Farrell_Jones_for_larger_families} implies that the map labelled $q$ is an isomorphism as well. The middle square in \cref{diag:bigger_kthy_computation_diagram} is induced by passing to homotopy groups in a pushout of spectra: it is given by applying the $\orbifoldfundamentalgroup$-homology theory $H^\orbifoldfundamentalgroup(-;\ncKthy)$ to the pushout \cref{lem:pushout_for_normaliser_family}. To deduce that the induced maps on vertical cokernels of homotopy groups are isomorphisms, it suffices to show that the vertical maps are injective on homotopy groups for $\ast \leq 0$. To see this, we unravel the equivalent vertical maps in \cref{diag:Kthy_computation_square} instead. They identify with the maps
	\begin{equation}
		\bigoplus_{F \in \maximal} \Kthy_{\ast}(\module{\bbZ N_{\orbifoldfundamentalgroup} F}^{\finitemodule}) \rightarrow \bigoplus_{F \in \maximal} \ncKthy_{\ast}(\module{\bbZ N_\orbifoldfundamentalgroup F}^{\finitemodule}) \hspace{3mm} \text{and} \hspace{3mm} \Kthy_\ast(\module{\bbZ \orbifoldfundamentalgroup}^\finitemodule) \rightarrow \ncKthy_\ast(\module{\bbZ\orbifoldfundamentalgroup}^\finitemodule).
	\end{equation}
	Both maps are injective for $\ast \leq 0$. Indeed, for any discrete group $V$, we can use the augmentation map $\bbZ V \rightarrow \bbZ$ to construct the composite
	\begin{equation}
		\Kthy_\ast (\module{\bbZ V}^\finitemodule) \rightarrow \ncKthy_\ast (\module{\bbZ V}^\finitemodule) \rightarrow \ncKthy_\ast(\module{\bbZ}^\finitemodule)
	\end{equation}
	and this composite is an isomorphism for $\ast \leq 0$, proving injectivity of the first map.
\end{proof}

The results \cref{thm:periodic_karoubi_structure_spectrum_vanishing} and \cref{thm:tateK_computation} are exactly the assumptions for \cref{lem:when_structure_spectra_are_truncated_L_theory} to apply, allowing us to deduce the following.

\begin{cor}
	\label{cor:structure_spectrum_assembly_map_is_isomorphism}
	Suppose that $\nielsspace$ admits a good decomposition with complement-quotient $(Q,\partial Q)$, and that $\orbifoldfundamentalgroup$ is a Farrell--Jones group without virtually cyclic subgroups of type II. Then the structure spectrum assembly map	
	\begin{equation}
		(\coconnquadL(\ast)) \otimes (Q,\partial Q) \xrightarrow{\simeq} \structure(Q,\partial Q)
	\end{equation}
	is an equivalence. In particular, if $(Q,\partial Q)$ is $d$-dimensional, $\Omega^{\infty +d+1} \structure(Q,\partial Q) = \ast$.
\end{cor}

Our method to identify $\structure(Q,\partial Q)$ can be used to compute some Whitehead groups and reduced projective class groups of integral group rings, which might be of independent interest. 

\begin{proof}[Proof of \cref{thm:whitehead_groups}]
	Consider the following diagram.
	\begin{equation}
		\begin{tikzcd}
			\coprod_{F \in \maximal} \Gamma/N_\Gamma F \times E \Gamma \ar[r] \ar[d, "q"] & E \Gamma \ar[d, "p"]\\
			\coprod_{F \in \maximal} \Gamma/N_\Gamma F \times E_\torsionfree \Gamma \ar[r] \ar[d] & E_\torsionfree \Gamma \ar[d]\\
			\coprod_{F \in \maximal} \Gamma/ N_\Gamma F \ar[r] & E_{\normfam} \Gamma
		\end{tikzcd}
	\end{equation}
	The map $H^\Gamma(\Gamma/H \times E \Gamma;\ncKthy) \rightarrow H^\Gamma(\Gamma/H;\ncKthy)$ is an equivalence whenever $H$ is a torsionfree subgroup by the Farrell--Jones conjecture for the regular ring $\bbZ$ \cite[Thm. 13.65(ii)]{lueckIC}.
	Hence, upon applying $H^\Gamma(-;\ncKthy)$, the maps labelled $p$ and $q$ become equivalences. Since the lower square is a pushout, this implies that $H^\Gamma(-;\ncKthy)$ sends the outer square to a pushout. Since $\vcyc \subset \normfam$, the map $H^\Gamma(E_\normfam \Gamma;\ncKthy) \rightarrow H^\Gamma(\Gamma/\Gamma;\ncKthy)$ is an equivalence. Unraveling, this yields that the following diagram of assembly maps
	\begin{equation}
		\label{eq:diag_whitehead_group_computations}
		\begin{tikzcd}
			\bigoplus_{F \in \maximal} BN_\Gamma F \otimes  \ncKthy(\bbZ)  \ar[r] \ar[d] &  B\Gamma \otimes \ncKthy(\bbZ)  \ar[d]\\
			\bigoplus_{F \in \maximal} \ncKthy(\bbZ N_\Gamma F) \ar[r] & \ncKthy(\bbZ \Gamma)
		\end{tikzcd}
	\end{equation}
	is a pushout. In particular, the map induced on vertical cofibres is an equivalence. Applying $\pi_1$ to this equivalence yields the claim about Whitehead groups, and applying $\pi_0$ yields the claim about reduced projective class groups.
\end{proof}

\begin{rmk}
	In the diagram \cref{eq:diag_whitehead_group_computations}, the ring $\bbZ$ can be replaced by a regular ring $R$, and it is still a pushout. Indeed, the additional input is that if $V$ is torsionfree, then the assembly map $BV \otimes \ncKthy(R) \rightarrow \ncKthy(RV)$ is an equivalence \cite[Thm. 13.65(ii)]{lueckIC}, so the proof applies verbatim. 
\end{rmk}

\subsection{The structure spectrum assembly map and injectivity of the transfer}

In this subsection we concern ourselves with proving the desired injectivity of the transfer map
\begin{equation}
	\pi_d \structure(Q,\partial Q) \rightarrow \pi_d\structure(C^{e},\partial C^{e}).
\end{equation}
Here $d$ is the dimension of the complement-quotient $(Q,\partial Q)$ as a Poincar\'e duality pair, so the degree in which the element $\tso(Q,\partial Q)$ lives. Our method is to exploit that the structure spectrum assembly map is an equivalence by \cref{cor:structure_spectrum_assembly_map_is_isomorphism}, and to compare with the transfer on singular homology. Recall that $\pi_0 \quadL(\ast) \simeq \bbZ$ induced by dividing the signature of a nondegenerate quadratic form over the integers by $8$. Using this isomorphism, we get a map of spectra
\begin{equation}
	\label{eq:map_from_singular_homology_to_coconnective_Lthy}
	\bbZ \simeq \tau_{\geq 0}(\tau_{\leq 0} \quadL(\ast)) \xrightarrow{u} \tau_{\leq 0} \quadL(\ast). 
\end{equation}
This produces a map of functors $\bbZ \otimes (-) \rightarrow \tau_{\leq 0} \quadL(\ast) \otimes (-)$ in $\func(\spcpairtrans,\spectra)$. Note that $\pi_i (\bbZ \otimes (Y,A))$ computes the singular homology $H_i(Y,A;\bbZ)$. Since $(Q,\partial Q)$ has vanishing singular homology in degrees $\ast \geq d+1$ by Poincar\'e--Lefschetz duality, a simple computation with the Atiyah-Hirzebruch spectral sequence proves the following.

\begin{cor}
	\label{cor:computation_of_structure_spectra}
	Under the assumptions of \cref{cor:structure_spectrum_assembly_map_is_isomorphism}, the structure spectrum of the complement-quotient $(Q,\partial Q)$ satisfies
	\begin{equation}
		\Omega^{\infty + d}\structure(Q,\partial Q) \simeq \Omega^{\infty + d}(\bbZ \otimes (Q,\partial Q)).
	\end{equation}
	In particular, $\pi_d \structure(Q,\partial Q) \cong H_d(Q,\partial Q)$ is an infinite cyclic group, and $\Omega^{\infty + d+1} \structure(Q,\partial Q) \simeq \ast$.
\end{cor}

The addendum $\Omega^{\infty + d+1} \structure(Q,\partial Q) \simeq \ast$ is interesting for the purpose of computing block homeomorphism groups. Our main application only concerns $\pi_d\structure(Q,\partial Q)$.

\begin{thm}
	\label{thm:transfer_injectivity}
	Suppose that $\nielsspace$ admits a finite good decomposition with complement-quotient $(Q,\partial Q)$ a $d$-dimensional Poincar\'e pair, and that $\orbifoldfundamentalgroup$ is a Farrell--Jones group that does not contain any virtually cyclic subgroups of type II. Then the transfer map
	\begin{equation}
		\pi_d \structure(Q,\partial Q) \rightarrow \pi_d \structure(C^{e},\partial C^{e})
	\end{equation}
	is injective.
\end{thm}

\begin{proof}
	Let $q \colon (C^{e},\partial C^{e}) \rightarrow (Q,\partial Q)$ denote the covering map. Recall that it is obtained by taking the quotient by the action of the finite group $G$. 
	We use the map $u$ of \cref{eq:map_from_singular_homology_to_coconnective_Lthy} to construct the following commuting diagram.
	\begin{equation}
		\begin{tikzcd}
			\pi_d \bbZ \otimes (Q,\partial Q) \ar[r] \ar[d, "q^\ast"] & \pi_d(\tau_{\leq 0} \quadL(\ast)) \otimes (Q,\partial Q) \ar[r] \ar[d, "q^\ast"]  &\structure(Q,\partial Q) \ar[d, "q^\ast"]  \\
			\pi_d \bbZ \otimes (C^{e},\partial C^{e}) \ar[r] \ar[d, "q_!"]  & \pi_d(\tau_{\leq 0} \quadL(\ast)) \otimes (C^{e},\partial C^{e}) \ar[r] \ar[d, "q_!"] &\structure(C^{e},\partial C^{e}) \ar[d, "q_!"] \\
			\pi_d \bbZ \otimes (Q,\partial Q) \ar[r] & \pi_d(\tau_{\leq 0} \quadL(\ast)) \otimes (Q,\partial Q) \ar[r] &\structure(Q,\partial Q)
		\end{tikzcd}
	\end{equation}
	Here, the maps $q^{\ast}$ denote the transfer map, provided by the wrong way functoriality in finite coverings, and $q_!$ denotes forward functoriality.
	By \cref{cor:computation_of_structure_spectra}, the upper horizontal composite and the lower horizontal composite are isomorphisms. Hence, to prove injectivity of the desired transfer map, it suffices to show that the left vertical composite is injective. By \cref{appendixprop:covering_and_transfer_for_singular_homology}, this composite is given by multiplication with the number of sheets of the covering $q$, which is exactly the order of the finite group $G$. But the group $\pi_d \bbZ \otimes (Q,\partial Q)$ is infinite cyclic, so multiplication with any nonzero number is injective. 
\end{proof}

Before moving on with our construction of group actions on aspherical manifolds, we briefly prove our application to block homeomorphisms of topological manifolds. 

\begin{proof}[Proof of \cref{thm:block_borel}]
	We note that the group $\orbifoldfundamentalgroup \coloneqq \pi_1(M_{hG})$ contains the fundamental group of a closed, nonpositively curved manifold as a finite index subgroup, hence it is a Farrell--Jones group. Since $G$ is assumed to have odd order, $\orbifoldfundamentalgroup$ does not contain elements of even order, and hence no virtually cyclic subgroups of type II.
	In total, $(N,\partial N)$ satisfies the assumptions of \cref{cor:computation_of_structure_spectra}, so $\Omega^{\infty + d+1}\structure(N,\partial N) \simeq \ast$.
	
	Write $\structure^s_{\mathrm{geo}}(N,\partial N)$ for the simple geometric structure space of $(N,\partial N)$, and $\structure^h_\mathrm{geo}(N,\partial N)$ for the geometric structure space. There is a map of spaces
	\begin{equation}
		\structure^s_{\mathrm{geo}}(N,\partial N) \rightarrow \structure_{\mathrm{geo}}(N,\partial N)
	\end{equation}
	which is an equivalence, provided that the map induced on Whitehead groups $\whiteheadgroup(\partial N) \rightarrow \whiteheadgroup(N)$ is an isomorphism. This holds in our case by \cref{thm:whitehead_groups}.
	Ranicki's main theorem on the total surgery obstruction identifies $\structure^h_\mathrm{geo}(N,\partial N) \simeq \Omega^{\infty + d + 1} \structure(N,\partial N) \simeq \ast$. Hence $\structure^s_\mathrm{geo}(N,\partial N) \simeq \ast$.
	
	Let $\hAut^s(N,\partial N) \subset \hAut(N,\partial N)$ denote the collection of components on the simple homotopy equivalences: homotopy equivalences $(f,\partial f) \colon (N,\partial N) \rightarrow (N,\partial N)$ such that both $f$ and $\partial f$ are simple. Then, from the $s$-cobordism theorem for manifold triads
	\begin{equation}
		\pi_k \hAut^s(N,\partial N)/ \widetilde{\mathrm{Homeo}}(N,\partial N) \simeq \pi_0 \structure^s_{\mathrm{geo},\partial_0}(N \times D^k,\partial N \times D^k) \simeq \pi_k \structure_{\mathrm{geo}}^{s}(N,\partial N).
	\end{equation}
	Here, $\partial_0$ is the pair $(N \times S^{k-1},\partial N \times S^{k-1})$, and $\pi_0 \structure^s_{\mathrm{geo},\partial_0}(N \times D^k,\partial N \times D^k)$ is Wall's simple structure set of the triad $(N \times D^k, \partial N \times D^k , N \times S^{k-1})$ relative to $N \times S^{k-1}$.
\end{proof}

\section{Gluing}

\label{sec:Gluing}

Suppose given a manifold pair $(\overline{N},\partial \overline{N})$ with a free action by a finite group $G$, and a manifold $L$ with a trivial $G$-action. Suppose further given an equivariant map $p \colon \partial \overline{N} \rightarrow L$. In the (homotopy) pushout of $G$-spaces
\begin{equation}
	\label{eq:pushout_gluing}
	\begin{tikzcd}
		\partial \overline{N} \ar[r] \ar[d, "p"] & \overline{N} \ar[d] \\
		L \ar[r] & X
	\end{tikzcd}
\end{equation}
the $G$-space $X$ is a $G$-Poincar\'e duality space, provided that $p$ has compact fibres which are equivalent to spheres after forgetting the group action. We want to study the question when the $G$-homotopy type $X$ is equivalent to  closed topological manifold with $G$-action. 

\subsection{Manifold approximate fibrations and teardrop neighborhoods}

As alluded to in the introduction, if $p$ would be homotopic to an equivariant sphere bundle projection, one could glue in the corresponding disc bundle. This condition is trivially fulfilled in the case where $L$ is a discrete set, but can usually not be expected. In this section, we show that we can glue $L$ and $(\overline{N},\partial \overline{N})$ under much weaker assumptions. The crucial notion is that of an approximate fibration, which we recall for continuous maps between euclidean neighborhood retracts (ENRs).

\begin{defn}
	A continuous map $p \colon E \rightarrow B$ between ENRs is an \textit{approximate fibration} if
	\begin{enumerate}
		\item it is a proper map;
		\item for each open subset $U \subset B$ the square
		\begin{equation}
			\begin{tikzcd}
				p^{-1}(U) \ar[r] \ar[d] & E \ar[d]\\
				U \ar[r] & B
			\end{tikzcd}
		\end{equation}
		is homotopy cartesian.
	\end{enumerate}
\end{defn}

The notion of approximate fibrations was introduced by Coram-Duvall \cite{approx}, and recently revisited in the context of higher topos theory in \cite{approximate}, where a few equivalent characterisations are given. In the situation of the pushout \cref{eq:pushout_gluing}, we can construct a manifold in the $G$-homotopy type $X$, provided there is an approximate fibration
\begin{equation}
	q \colon G \backslash \partial C \times \bbR \rightarrow L \times \bbR
\end{equation}
which induces the map $G \backslash p \colon G \backslash \partial C \rightarrow L$ up to homotopy. The key is the following construction which was brought to our attention by Shmuel Weinberger.

\begin{constr}[\cite{HughesTaylorWilliams}, Sec. 3]
	\label{constr:teardrop}
	Let $E$ and $B$ be topological spaces, and let $p \colon E \rightarrow B \times \bbR$ be a map. Then the \textit{teardrop construction on $p$} is the topological space $E \cup_p B$ where
	\begin{enumerate}
		\item the underlying set of $E \cup_p B$ is the union of $E$ and $B$;
		\item the topology is minimal with the property that $E \subset E \cup_p B$ is an open embedding and that the map
		\[ E \cup_p B \rightarrow B \times (-\infty, +\infty], \hspace{3mm}  x \mapsto \begin{cases}
			p(x) &\text{ for $x \in E$};\\
			(x,+\infty) &\text{ for $x \in B$};
		\end{cases} \]
		is continous, where $(-\infty,\infty]$ is topologised to be homeomorphic to $(0,1]$.
	\end{enumerate}
\end{constr}

\begin{rmk}
	\label{rmk:teardrop_and_controlled_homeomorphism_type}
	\cite[Thm. 2.2.]{HughesTaylorWeinberger} says that
	the homeomorphism type of the teardrop construction on $p$ only depends on the \textit{controlled homeomorphism type} of $E$ over $B$.
\end{rmk}

\begin{prop}
	Let $p \colon E \rightarrow B \times \bbR$ be an approximate fibration. Assume that
	\begin{enumerate}
		\item both $E$ and $B$ are topological manifolds;
		\item the homotopy fibres of $p$ are equivalent to spheres.
	\end{enumerate}
	Then the teardrop $E \cup_p B$ on $p$ is a topological manifold.
\end{prop}

\begin{proof}
	We check that $E \cup_p B$ is locally euclidean, the Hausdorff property and second countability are easily verified. Since $E$ is locally euclidean and $E \subset E \cup_p B$ is open, it suffices to check  that each $x \in B \subset E \cup_p B$ has a euclidean neighborhood. Observe that for an open neighborhood $x \in U \subset B$, the teardrop on the approximate fibration $p\colon p^{-1}(U \times \bbR) \rightarrow U \times \bbR$ is an open neighborhood of $x \in E \cup_p B$. So we may assume for the rest of the proof that $B$ is a euclidean space $\bbR^k$. The teardrop of the projection map $\pi \colon S^{d-k-1} \times \bbR^{k} \times \bbR \rightarrow \bbR^{k} \times \bbR$ is homeomorphic to $\bbR^d$. Hence, by \cref{rmk:teardrop_and_controlled_homeomorphism_type} it suffices to argue that any approximate fibration $q \colon E' \rightarrow \bbR^{k+1}$ with homotopy fibre $S^{d-k-1}$ is controlled homeomorphic to this projection $\pi$. This is shown in \cite{Chapman81}.
\end{proof}

\subsection{$K$-theoretic obstructions to approximate fibering}

To apply the results of the preceding subsection in practice, we need techniques to construct approximate fibrations over the product of a closed manifold $M$ with $\bbR$. One way to produce approximate fibrations over $M \times \bbR$ is to construct an approximate fibration over $M \times S^1$, and pulling back along the infinite cyclic cover $M \times \bbR \rightarrow M \times S^1$. This motivates the following.

\begin{defn}
	A closed aspherical manifold $M^k$ has the \textit{approximate fibering property} if for each map $ p \colon E^d \rightarrow M^k$, $d \geq 5$ which is surjective on fundamental groups and has compact homotopy fibres, the map
	\[ p \times S^1 \colon E \times S^1 \rightarrow M \times S^1 \]
	is homotopic to an approximate fibration.
\end{defn}

\begin{thm}[Farrell--L\"uck--Steimle]
	\label{thm:FLS}
	Let $M$ be an aspherical manifold whose fundamental group is word-hyperbolic and which admits a PL-structure. Then $M$ has the approximate fibering property.
\end{thm}

Since this theorem does not appear verbatim in \cite{FLS}, we deduce it from the results therein below. We note that we believe that the results of this and the preceding subsection justify revisiting the theory of approximate fibrations, and to understand analogs of \cref{thm:FLS} for more general base spaces. For example, it would be interesting to remove the PL-hypothesis, or to have a useful obstruction theory for non-aspherical base spaces.

\begin{recollect}
	Compact manifolds are ENRs. A homotopy equivalence $h \colon X \rightarrow Y$ of compact ENRs has a Whitehead torsion $\tau(h) \in \whiteheadgroup(X)$ in the Whitehead group of $X$. For a homotopy equivalence  $h \colon X \rightarrow Y$ of ENRs, and another ENR $Z$ with zero Euler characteristic, one has $0=\tau(h \times Z \colon X \times Z \rightarrow Y \times Z)$, see \cite[Ch. 2]{lueckIC}.
\end{recollect}

\begin{proof}[Proof of \cref{thm:FLS}]
	By \cite{FLS}, the obstruction to finding an approximate fibration in the homotopy class of $p \times S^1$ is constructed as follows.
	\begin{enumerate}
		\item Construct any factorisation of $p \times S^1$ 
		\begin{equation}
			E \times S^1 \xrightarrow{h} P \xrightarrow{q} M \times S^1
		\end{equation}
		where $h$ is a homotopy equivalence, $P$ a compact ENR and $q$ an approximate fibration.
		\item Take the Whitehead torsion $\tau(h) \in \whiteheadgroup(E \times S^1)$ of the homotopy equivalence $h$.
		\item Pass to a certain quotient of the Whitehead group, in which the constructed obstruction does not depend on any of the choices made above.
	\end{enumerate}
	We note that it is possible to pick $P$ such that $\tau(h) =0 \in \whiteheadgroup(E \times S^1)$. Namely, pick a factorisation of $p \colon E \rightarrow M$ into
	\[ E \xrightarrow{w} Q \xrightarrow{r} M \]
	where $w$ is a homotopy equivalence, $Q$ an ENR and $r$ an approximate fibration. This is possible by \cite[Thm. 4.1.]{FLS}. Setting $P = Q \times S^1$, $h = w \times S^1$ and $q$ the composite of $r \times S^1$ with the projection $M \times S^1 \rightarrow M$, we note that $\tau(h) = \tau(w \times S^1) = 0$.  Since the obstruction to finding an approximate fibration within the homotopy class of $p \times S^1$ is the image of $\tau(h)$ in a quotient of $\whiteheadgroup(E \times S^1)$, we see that it vanishes, as desired.
\end{proof}

\section{The main theorem}

In this section, we combine the results of the previous sections to prove \cref{thm:main_theorem}, and in fact a slight refinement, \cref{thm:refined_main_theorem}. 

\subsection{Manifold structures on the complement}

We start with a preliminary step, motivated by the following. Suppose we have a $G$-space $X$ equipped with the structure of a finite semifree isovariant $G$-Poincar\'e space, with complement $(C,\partial C)$ and complement quotient $(Q,\partial Q)$. In \cref{thm:transfer_injectivity} we showed that in the situation relevant for \cref{thm:main_theorem}, $\tso(C^{e},\partial C^{e}) = 0$ implies $\tso(Q,\partial Q) = 0$. Here, we give the relevant technique to show that $\tso(C^{e},\partial C^{e}) = 0$ in our situation. The main result \cref{lem:tso_substraction_lemma} deduces that from the fact that $(C,\partial C)$ is ``embedded" in the ambient manifold $M$, for which we want to solve the Nielsen realisation problem.

One way to show that $\tso(C^{e},\partial C^{e}) = 0$ is to construct an explicit manifold structure using embedding theory \cref{rmk:alternative_construction}. We give an algebraic argument instead. The results in this subsection are closely related to a result of Wall \cite[Cor. 11.3.1]{scm}. 

\begin{lem}
	\label{lem:tso_image_in_relative_term}
	Consider a pushout of finite spaces
	\begin{equation}
		\begin{tikzcd}
			\partial C \ar[r] \ar[d, "p"] & C \ar[d, "q"]\\
			U \ar[r] & X
		\end{tikzcd}
	\end{equation}
	in which $(C,\partial C)$ and $(U,\partial C)$ are $d$-dimensional oriented Poincar\'e duality pairs. Then $X$ naturally is an oriented Poincar\'e duality space, and in the diagram
	\begin{equation}
		\label{eq:diagram_of_structure_spectra}
		\pi_d \structure(X) \rightarrow \pi_d \structure(X,U) \leftarrow \pi_d \structure(C,\partial C)
	\end{equation}
	the elements $\tso(X)$ and $\tso(C,\partial C)$ have equivalent images.
\end{lem}

\begin{proof}
	See \cite{Wall} for the statement that $X$ is Poincar\'e, automatically of the same dimension $d$ as $(C,\partial C)$.
	The maps of \cref{eq:diagram_of_structure_spectra} fit into a commuting diagram as follows.
	\begin{equation}
		\begin{tikzcd}
			\pdbord_d(X) \ar[r] \ar[d] & \pdbord_d(X,U) \ar[d] & \pdbord_d(C,\partial C) \ar[l] \ar[d]\\
			\pi_d\structure(X) \ar[r] & \pi_d\structure(X,U) & \ar[l] \pi_d\structure(C,\partial C)
		\end{tikzcd}
	\end{equation}
	The total surgery obstructions $\tso(X)$ and $\tso(C,\partial C)$ are the images of the bordism classes $\id \colon X \rightarrow X$ and $\id \colon (C,\partial C) \rightarrow (C,\partial C)$ under the vertical maps. So to prove the statement, it suffices to exhibit a Poincar\'e bordism of the maps $(X,\emptyset) \rightarrow (X,U)$ and $(C,\partial C) \rightarrow (X,U)$. Such a bordism is given by the explicit diagram
	\begin{equation}
		\begin{tikzcd}
			X \ar[r]& X & C \ar[l] \\
			\emptyset \ar[r] \ar[u] & U \ar[u]  & \partial C \ar[l] \ar[u] 
		\end{tikzcd}
	\end{equation}
	of Poincar\'e duality pairs with a reference map to the pair of spaces $(X,U)$.
\end{proof}

\begin{lem}
	\label{lem:tso_substraction_lemma}
	In the situation of \cref{lem:tso_image_in_relative_term}, if $p$ is a $\pi$--$\pi$-equivalence, $\tso(X) =0$ implies $\tso(C,\partial C) = 0$.
\end{lem}

\begin{proof}
	The statement is a consequence of \cref{lem:tso_image_in_relative_term} and the fact that $\structure(C,\partial C) \rightarrow \structure(X,U)$ is an equivalence. To see the latter, it suffices to show that both the maps
	\begin{equation}
		\quadL(\ast)_{\geq 1} \otimes (C,\partial C) \rightarrow \quadL(\ast)_{\geq 1} \otimes (X,U) \hspace{3mm} \text{and} \hspace{3mm} \quadL(C,\partial C) \rightarrow \quadL(X,U)
	\end{equation}
	are equivalences. The first is a consequence of excision. For the second, since $\Lthy(Y,B) \simeq \cofib(\quadL(B) \rightarrow \quadL(Y))$ we note that it follows if we can justify that 
	\begin{equation}
		\begin{tikzcd}
			\quadL(\partial C) \ar[r] \ar[d, "\quadL(p)"] & \quadL(C) \ar[d, "\quadL(q)"]\\
			\quadL(U) \ar[r] & \quadL(X)
		\end{tikzcd}
	\end{equation}
	is cocartesian. By assumption, $p$ is a $\pi$--$\pi$-equivalence, and hence so is $q$. But $\quadL(-)$ inverts $\pi$--$\pi$-equivalences, so both vertical maps in the above square are equivalences, which implies that it is cocartesian.
\end{proof}

\begin{rmk}
	\label{rmk:alternative_construction}
	This remark presents an alternative construction of a manifold structure on $(C,\partial C)$ under the hypothesis that $M^{hC_p}$ in \cref{thm:main_theorem} is equivalent to a $k$-dimensional topological manifold with $d \geq 2k + 4$. This shows that $\tso(C,\partial C) = 0$ by a geometric argument, under a stronger codimension assumption. We include it to illustrate how the algebraic argument presented above leads to a stronger conclusion.
	
	In the situation of \cref{thm:main_theorem} pick a closed topological manifold $L$ homotopy equivalent to $M^{hC_p}$, and let $k$ be the dimension of $L$. Bringing the map $ L \rightarrow M$ induced by the map $M^{hC_p} \rightarrow M$ into general position, we may assume that $L$ is a locally flat embedded submanifold, using that $2k < d$. We claim that $M\setminus L$ is the interior of a compact manifold by checking that Siebenmann's end theorem applies. 
	
	First, the embedding $L \subset M$ defines a Poincar\'e embedding of $L$ into $M$. The assumption $d \geq 2k + 4$ implies by \cite[Cor. B]{KleinEmbII} that this Poincar\'e embedding is concordant to the embedding provided by the isovariant structure on $M$. In  particular, the end of $M \setminus L$ identifies with $\partial C$. Since $\partial C$ is a finite space, we conclude that indeed Siebenmann's end theorem applies to show that $M \setminus L$ is the interior of a compact manifold $\overline{M \setminus L}$. Then $(C,\partial C) \simeq (\overline{M \setminus L} ,\partial \overline{M \setminus L})$.
\end{rmk}

\subsection{Putting together}

\begin{thm}
	\label{thm:refined_main_theorem}
	Let $M^d$ be an oriented aspherical $d$-manifold with Farrell--Jones fundamental group $\pi$, and $\alpha \colon G \rightarrow \hAut(M)$ an $E_1$-map. Assume that
	\begin{enumerate}
		\item the dimension $d$ is at least six;
		\item the map $\alpha$ takes values in orientation-preserving automorphisms;
		\item the orbifold fundamental group $\orbifoldfundamentalgroup = \pi_1 (M_{hG})$ contains no virtually cyclic subgroups of type II;
		\item the Borel $G$-space $\nielsspace$ associated to $(M,\alpha)$ admits the structure of a finite semifree isovariant $G$-Poincar\'e duality space;
		\item the space $M^{hG}$ has the homotopy type of a closed manifold $L$ of dimension at most $d-3$, which has the approximate fibering property.
	\end{enumerate}
	Then $\alpha$ refines to a Borel $G$-action on $M$.
\end{thm}

\begin{proof}
	Let $(C,\partial C)$ denote the complement of the finite semifree isovariant $G$-Poincar\'e structure and $(Q,\partial Q)$ the complement-quotient. The orientation on $M$ equips $(C,\partial C)$ with an orientation, which descends to $(Q,\partial Q)$, since $\alpha$ acts orientation-preserving. We show that $\tso(Q,\partial Q) = 0$. The image of $\tso(Q,\partial Q)$ under the transfer map $\pi_d \structure(Q,\partial Q) \rightarrow \pi_d\structure(C^{e},\partial C^{e})$ is $\tso(C^{e},\partial C^{e})$. Since $\tso(M) = 0$, \cref{lem:tso_substraction_lemma} shows that $\tso(C^{e},\partial C^{e}) = 0$. The transfer is an injective map by \cref{thm:transfer_injectivity}, so we actually deduce that $\tso(Q,\partial Q) = 0$.
	
	Since $d \geq 6$, we can apply Ranicki's theorem (recalled in \cref{thm:Ranicki_tso}) to deduce the existence of a manifold structure $(N,\partial N) \simeq (Q,\partial Q)$. Pulling back along the $G$-cover $(C,\partial C) \rightarrow (Q,\partial Q)$, we deduce the existence of a $G$-manifold structure $(\overline{N},\partial \overline{N}) \simeq (C,\partial C)$.
	
	Since $L$ has the approximate fibering property, we may find an approximate fibration in the homotopy class $\partial N \times S^1 \rightarrow L \times S^1$. Pulling back along the infinite cyclic cover of $S^1$ and precomposing with the $C_p$-cover of $\partial N$, we get an approximate fibration
	\begin{equation}
		p \colon \partial \widetilde{N} \times \bbR \rightarrow L \times \bbR
	\end{equation}
	which is also an equivariant map. Hence, the teardrop construction $\partial \widetilde{N} \times \bbR \cup_p L$ is a topological manifold with $G$-action. Using an equivariant collar we get an embedding $\partial \widetilde{N} \times \bbR \hookrightarrow \widetilde{N} \setminus \partial \widetilde{N}$.
	
	We define
	\begin{equation}
		M' \coloneqq (\widetilde{N} \setminus \partial \widetilde{N}) \cup_{\partial \widetilde{N} \times \bbR} \partial \widetilde{N} \times \bbR \cup_p L.
	\end{equation}
	Then $M'$ is a topological manifold with $G$-action. By construction, there is a $G$-equivariant homotopy equivalence $k \colon M' \rightarrow \nielsspace$. By the Borel conjecture for word-hyperbolic groups, the map $k\colon M' \rightarrow M$ is homotopic to a homeomorphism $h \colon M' \rightarrow M$.
	Consider the following diagram of $E_1$-groups, where the maps $c_h$ denote conjugation with the homeomorphism $h$, and $c_k$ for conjugation with $k$, and $\mathrm{Homeo}^{\delta}(-)$ denotes the homeomorphism group, with the discrete topology.
	\begin{equation}
		\begin{tikzcd}
			G \ar[r, "\mathrm{act}"] & \mathrm{Homeo}^\delta(M') \ar[r, "u_{M'}"] \ar[d, "c_h"] & \hAut(M') \ar[d, "c_h"] \ar[dr, "c_k"]& \\
			& \mathrm{Homeo}^{\delta}(M) \ar[r, "u_M"] & \hAut(M) \ar[r, "\simeq"] &\hAut(\nielsspace^{e})
		\end{tikzcd}
	\end{equation}
	The map $ \mathrm{act}$ is induced by the $G$-action on $M'$, and the square commutes.
	Since $h$ and $k$ are homotopic, the maps $c_h, c_k \colon \hAut(M') \rightarrow \hAut(M)$ are equivalent, showing that also the right triangle commutes. Since $c_k \circ u_{M'} \circ \mathrm{act}$ identifies with the $E_1$-map $\alpha$, we see that $c_h \circ \mathrm{act}$ indeed gives a $C_p$-action refining $\alpha$.
\end{proof}

\begin{proof}[Proof of \cref{thm:main_theorem}]
	Note that torsionfree word-hyperbolic groups which are not cyclic have trivial centre, so an $E_1$-map $\alpha \colon C_p \rightarrow \hAut(M)$ is the same as a group map $C_p \rightarrow \Out(\pi)$.
	We verify that the assumptions of \cref{thm:main_theorem} imply the assumptions of \cref{thm:refined_main_theorem}. 
	Word-hyperbolic groups are Farrell--Jones groups. If $p$ is odd, then $\alpha$ automatically acts orientation-preserving. If $\orbifoldfundamentalgroup$ had a virtually cyclic subgroup of type II, then $\orbifoldfundamentalgroup$ would possess elements of even order. But every finite subgroup of $\orbifoldfundamentalgroup$ is isomorphic to a subgroup of $C_p$, and none of them have even order. The existence of the structure of a finite semifree isovariant $C_p$-Poincar\'e space is provided by \cref{thm:existence_of_finite_isovariant_structures}. The dimension assumption of a PL-manifold homotopy equivalent to $M^{hC_p}$ implies that that dimension is at most $d-3$, since $(d-2)/2 \leq d-3$ holds for $d\geq 6$. Finally, word-hyperbolicity of $\orbifoldfundamentalgroup$ implies that the components of $L$ are aspherical PL-manifolds, which have the approximate fibering property by \cref{thm:FLS}.
\end{proof}

\begin{recollect}
	\label{recollect:Weyl_group_of_finite_subgroup_again_hyperbolic}
	If $\Gamma$ is a word-hyperbolic group, and $F \subset \Gamma$ a finite subgroup, then the Weyl group $W_\Gamma F$ is word-hyperbolic as well. Indeed, by \cite[Prop. III.$\Gamma$.3.9]{BridsonHaefliger}, the centraliser $C_\Gamma(F)$ of $F$ in $\Gamma$ is word-hyperbolic. Furthermore, if $n \in \Gamma$ normalises $F$, then conjugation by $n$ restricts to an automorphism of $F$. This action is trivial if and only if $n$ centralises $F$. This gives rise to an exact sequence
	\begin{equation}
		1 \rightarrow C_\Gamma(F) \rightarrow N_\Gamma(F) \rightarrow \mathrm{Aut}(F).
	\end{equation}
	Being word-hyperbolic is inherited by finite-index subgroups, finite-index overgroups and quotients by finite subgroups.
	Hence, since $F$ and $\mathrm{Aut}(F)$ are  finite, $N_\Gamma(F)$ and $W_\Gamma F$ are word-hyperbolic as well.
\end{recollect}

\appendix

\section{Group theory}

Let $\pi$ be a torsionfree group and let $G$ be a nontrivial finite group. For this section, we fix an extension
\begin{equation}
	1 \rightarrow \pi \rightarrow \Gamma \rightarrow G \rightarrow 1.
\end{equation}
A subgroup $F \subset \Gamma$ will be called \textit{full} if the composite $F \subset \Gamma \rightarrow G$ is an isomorphism.  In this appendix, we make a few group-theoretic observations about such extensions, in particular relating to the universal space for the family of finite subgroups $E_\finite \Gamma$.  Write $\maximal$ for a set of representatives of the conjugacy classes of full subgroups of $\Gamma$.

\begin{appendixlemma}
	\label{appendixlem:group_theory}
	Let $F \subset \Gamma$ be a full subgroup. Then:
	\begin{enumerate}
		\item every finite subgroup of $N_\Gamma F$ is contained in $F$;
		\item we have $N_\Gamma (N_\Gamma F) = N_\Gamma F$;
		\item for each full subgroup $H \subset \Gamma$, the $\Gamma$-set $\coprod_{F \in \maximal} \Gamma / N_\Gamma F$ has a unique $H$-fixed point.
	\end{enumerate}
\end{appendixlemma}

\begin{proof}
	For the first part, if $H \subset N_\Gamma F$ is not contained in $F$, then we see that $H$ and $F$ generate a subgroup $U$ in $\Gamma$ that is not contained in $F$. The composite $U \subset \Gamma \rightarrow G$ cannot be injective, so $U \cap \pi \neq 1$. Furthermore, $U$ is finite as $H$ normalises $F$, so that each $u \in U$ may be written as $hf$ with $h \in H$ and $f \in F$. This is a contradiction, since $\pi$ is torsionfree.
	The second part follows from the first part: if $g$ normalises $N_\Gamma F$, then $g^{-1} F g$ is finite, hence equal to $F$.
	
	For the third part, we compute that $(\Gamma/N_\Gamma F)^H$ is the set of cosets $gN_\Gamma F$ with $g^{-1} H g \subset N_\Gamma F$. By the first part, $g^{-1} H g = F$. If $h$ also satisfies $h^{-1} H h = F$, we have $g^{-1} h F h^{-1} g = g^{-1} H g = F$, so $h^{-1} g \in N_\Gamma F$ and $gN_\Gamma F = hN_\Gamma F$.
\end{proof}

\begin{appendixlemma}
	\label{appendixlemma:characterisations_of_property_M}
	The following are equivalent.
	\begin{enumerate}
		\item Every nontrivial finite subgroup of $\Gamma$ is contained in a unique full subgroup.
		\item The universal space $E_\finite \Gamma$ admits a model with cells of isotropy type $\{\Gamma/F\}_{F \in \maximal}$ or $\Gamma/e$.
		\item The $\Gamma$-set of path components of $(E_\finite \Gamma)^{>1}$ is $\Gamma$-isomorphic to $\coprod_{F \in \maximal} \Gamma/N_\Gamma F$.
		\item The square
		\begin{equation}
			\begin{tikzcd}
				\coprod_{F \in \maximal} \induct_{N_\Gamma F}^\Gamma EN_\Gamma F \ar[r] \ar[d] & E\Gamma \ar[d] \\
				\coprod_{F \in \maximal} \induct_{N_\Gamma F}^\Gamma E_\finite N_\Gamma F \ar[r] & E_\finite \Gamma
			\end{tikzcd}
		\end{equation}
		is cocartesian.
	\end{enumerate}
\end{appendixlemma}

\begin{proof}
	To prove (4) $\implies$ (2) we observe that a model for $E_\finite N_\Gamma F$ is given by restricting the action of $W_\Gamma F$ on $EW_\Gamma F$ along the surjection $N_\Gamma F \rightarrow W_\Gamma F$ with finite kernel. So $E_\finite N_\Gamma F$ can be built from cells of isotropy type $N_\Gamma F/ F$. Now observe that $\induct_{N_\Gamma F}^\Gamma (N_\Gamma F/F) = \Gamma/F$.
	The implication (1) $\implies$ (4) is a direct computation on fixed points, or follows from \cite[Cor. 2.11]{LueckWeiermann12}.
	
	To see that (2) $\implies$ (3), we construct two maps
	\begin{equation}
		\coprod_{F \in \maximal} \induct_{N_\Gamma F}^\Gamma E_\finite N_\Gamma F \xrightarrow{f}  E_\finite \Gamma^{>1} \xrightarrow{p} \coprod_{F \in \maximal} \Gamma/N_\Gamma F.
	\end{equation}
	The map $f$ is the unique $\Gamma$-map. To construct the second map, observe that each $\Gamma$-set with full isotropy groups uniquely maps to $\coprod_{F \in \maximal} \Gamma/N_\Gamma F$. By elementary obstruction theory, every $\Gamma$-space constructed out of cells with full isotropy groups has a unique map to $\coprod_{F \in \maximal} \Gamma/N_\Gamma F$. Thus, by assumption we have a unique map $p$, and we note that the composite $pf$ induces an isomorphism after applying the functor $\pi_0(-)^{e}$. It is left to show that $f$ induces a surjection after applying $\pi_0(-)^{e}$ to deduce $(3)$. 
	The assumption guarantees that each $x \in \pi_0(-)^{e}$ is in the image of the unique $\Gamma$-map $\Gamma/F \rightarrow E_\finite \Gamma^{>1}$ after applying $\pi_0(-)^{e}$ for some $F \in \maximal$. Each such map lifts over $f$ by inducting the unique $N_\Gamma F$-map $N_\Gamma F/F \rightarrow E_\finite N_\Gamma F$.
	
	To see that (3)$\implies$(1) let $H \subset \Gamma$ be a nontrivial finite subgroup. There is a unique $\Gamma$-map $\Gamma/H \rightarrow (E_\finite \Gamma)^{>1}$. The assumption implies that $H \subset N_\Gamma F$ for some full subgroup. By \cref{appendixlem:group_theory}, $H\subset F$. If there is a different full subgroup $F'$ containing $H$, then the postcomposition of the projection maps $\Gamma/H \rightarrow \Gamma/F$ and $\Gamma/H \rightarrow \Gamma/F'$ with the unique $\Gamma$-map to $\coprod_{F \in \maximal} \Gamma/N_\Gamma F$ sends $eH$ to different points. But the composite factors over $(E_\finite \Gamma)^{>1}$, and two $\Gamma$-maps $\Gamma/H \rightarrow (E_\finite \Gamma)^{>1}$ are homotopic.
\end{proof}

\begin{appendixlemma}
	\label{appendixlemma:normalisers_and_virtually_cyclic_subgroups}
	Under the equivalent conditions of \cref{appendixlemma:characterisations_of_property_M}, the group $\Gamma$ satisfies the following.
	\begin{enumerate}
		\item When $H$ is a nontrivial finite subgroup and $F$ the unique full subgroup containing it, then $N_\Gamma H \subset N_\Gamma F$.
		\item If $V$ is an infinite, non-torsionfree virtually cyclic subgroup of type I, $V_{\mathrm{max}}$ its maximal finite subgroup, then $V$ is contained in the normaliser of the unique full subgroup containing $V_{\mathrm{max}}$.
	\end{enumerate}
\end{appendixlemma}

\begin{proof}
	Let $H$ be a nontrivial finite subgroup, and $g\in \Gamma$ so that  $g^{-1} H g = H$. Then if $H \subset F$ and $F$ is full, also $H \subset g^{-1}F g$ and $g^{-1} F g$ is full. Therefore, $F = g^{-1} F g$ so $g \in N_\Gamma F$.
	For the second point, note that $V \subset N_\Gamma V_{\mathrm{max}}$, and that $V_{\mathrm{max}}$ is nontrivial. Hence, the assertion follows from the first part of the lemma.
\end{proof}

\begin{appendixlemma}
	\label{appendixlemma:computing_fixed_points}
	Under the equivalent conditions of \cref{appendixlemma:characterisations_of_property_M}, writing $X \coloneqq \pi \backslash E_\finite \Gamma$, we have preferred equivalences
	\begin{equation}
		\coprod_{F \in \maximal} B W_\Gamma F \simeq X^G \simeq X^{>1}.
	\end{equation}
	Furthermore, the map $(X^G)_{hG} \rightarrow (X^e)_{hG}$ is equivalent to the map
	\begin{equation}
		\coprod_{F \in \maximal} BN_\Gamma F \rightarrow B\Gamma
	\end{equation}
	induced by the inclusions $N_\Gamma F \subset \Gamma$.
\end{appendixlemma}

\begin{proof}
	Write $N_\Gamma^{\pi}F = N_\Gamma F \cap \pi$. The composite of the inclusion with the projection
	\begin{equation}
		N_\Gamma^\pi F \xhookrightarrow{i} N_\Gamma F \xrightarrow{q} W_\Gamma F 
	\end{equation}
	is an isomorphism, and $q^* E W_\Gamma F \simeq E_\finite N_\Gamma F$. So, using (4) in \cref{appendixlemma:characterisations_of_property_M} we compute
	\begin{equation}
		X^{>1} \simeq \pi \backslash \coprod_{F \in \maximal} \induct_{N_\Gamma F}^\Gamma E_\finite N_\Gamma F \simeq \pi \backslash \coprod_{F \in \maximal} \induct_{N_\Gamma F}^\Gamma q^* E W_\Gamma F \simeq \coprod_{F \in \maximal} N_\Gamma^\pi F \backslash q^* E W_\Gamma F.
	\end{equation}
	It is clear that the residual $G$-action in the right hand term is trivial, and that the quotient is exactly $BW_\Gamma F$. For the second assertion, we similarly apply $(-)_{h\Gamma}$ to the upper row of the diagram in (4) of \cref{appendixlemma:characterisations_of_property_M}, which is easily seen to identify with the desired map.
\end{proof}

\section{Pairs of spaces with transfers}

\label{sec:pairs_of_spaces}

The purpose of this appendix is to provide the technical background on the category $\spcpairtrans$ sketched in \cref{subsec:Lthy_for_pairs}. For adequate categories $\category{C}$ we give a procedure for the construction of functors
\[ F \colon \spcpairtrans \rightarrow \category{C}. \]
We rely on general categorical technology from the existing literature. First, we give some generalities.  Recall that given a category $\category{X}$ with pullbacks, and two classes of morphisms, one being the class of \textit{left} morphisms $\category{X}_{\leftclass}$, and one the class of \textit{right} morphisms $\category{X}_{\rightclass}$ which are closed under pullbacks and composition, we can form a category of \textit{correspondences} $\Corr(\category{X},\category{X}_\leftclass,\category{X}_\rightclass)$. A morphism $X \dashrightarrow Y$ in this category is a span $X \xleftarrow{p} E \xrightarrow{f} Y$ with $p \in \category{X}_\leftclass$ and $f \in \category{X}_\rightclass$.  For a comprehensive treatment, see \cite[Sec. 2]{HHLN23}. In particular observe that they construct a natural equivalence
\begin{equation}
	\label{eq:opposite_cats_of_spans}
	\Corr(\category{X},\category{X}_\leftclass,\category{X}_\rightclass)\op \simeq \Corr(\category{X},\category{X}_\rightclass,\category{X}_\leftclass).
\end{equation}
Set $\spc^{2} = \func(\Delta^1,\spc)$ to be the category of pairs of spaces. We write generic objects as $(X,A) \in \spc^2$. 

\begin{enumerate}
	\item Write the classes of equivalences, fold maps, finite coverings and all maps as
	\[\equi \subset \fold \subset \fincov  \subset \all.\] 
	The class of fold maps is the class of maps equivalent to $(X,A)^{\cup n} \rightarrow (X,A)$ for some finite integer $n$, and finite coproducts of such maps for varying $n$. All these classes are stable under pullback and composition. We get a chain of wide subcategory inclusions
	\begin{equation*}
		\spc^{2} \simeq \Corr(\spc^{2}, \equi,\all) \rightarrow \Corr(\spc^{2},\fold, \all) \rightarrow \spcpairtrans \coloneqq \Corr(\spc^{2},\fincov,\all).
	\end{equation*}
	\item The category $\spc^2$ is \textit{extensive} in the sense of \cite[Def. 2.3.]{BH21}. That is, it admits finite coproducts, coproducts are disjoint and finite coproduct decompositions are closed under pullbacks. 
	\item We write $\finpairs \subset \spc^2$ and $\finpairstrans \subset \spcpairtrans$ for the full subcategories spanned by pairs of spaces $(T,S)$ for which $S \rightarrow T$ is an injection of finite sets.
	\item The categories $\spcpairtrans$ and $\finpairstrans$ are semiadditive, and the functors $\spc^{2} \rightarrow \spcpairtrans$ as well as $\finpairs \rightarrow \finpairstrans$ preserve finite coproducts, see \cite[Lem. C.3]{BH21}. 
\end{enumerate}

\begin{appendixlemma}
	\label{lem:construction_of_functors_with_transfers}
	Let $\category{C}$ be a cocomplete category. The restriction functor
	\begin{equation}
		\func(\spcpairtrans,\category{C}) \rightarrow \func(\finpairstrans,\category{C})
	\end{equation}
	induces an equivalence between
	\begin{enumerate}
		\item those $F \colon \spcpairtrans \rightarrow \category{C}$ whose restriction to $\spc^2$ commutes with colimits;
		\item those $G \colon \finpairstrans \rightarrow \category{C}$ whose restriction to $\finpairs$ commutes with coproducts.
	\end{enumerate}
\end{appendixlemma}

\begin{proof}
	We observe that $\finpairs \subset \spc^{2}$ satisfies the following: if $f \colon (X,A) \rightarrow (T,S)$ is a morphism in $\fincov$, and $(T,S) \in \finpairs$, then also $(X,A) \in \finpairs$.
	By dualising \cite[Prop. C.21]{BH21}, this observation implies that the horizontal restriction maps depicted below fit into a left adjointable diagram.
	\begin{equation}
		\label{eq:spaces_with_transfers_square}
		\begin{tikzcd}
			\func(\spcpairtrans,\category{C}) \ar[r, "\widetilde{R}"'] \ar[d, "p"] & \func(\finpairstrans,\category{C}) \ar[d, "q"] \ar[l, bend right, dashed, "\widetilde{L}"']\\
			\func(\spc^2,\category{C}) \ar[r, "R"'] & \func(\finpairs,\category{C})\ar[l, bend right, dashed, "L"']
		\end{tikzcd}
	\end{equation}
	Since $\widetilde{L}$ is a left Kan extension along a fully faithful functor, it is fully faithful. It restricts to a fully faithful functor $\widetilde{L}^{c}$ on the subcategory of those $G \colon \finpairstrans \rightarrow \category{C}$ such that $qG$ commutes with coproducts. We argue that the essential image of $\widetilde{L}^{c}$ consists of exactly those $F$ for which $pF$ commutes with colimits.
	
	So let $F \colon \spcpairtrans \rightarrow \category{C}$ be so that $pF$ preserves colimits. Then $RpF \simeq q \widetilde{R} F$ commutes with finite coproducts. Since $q$ is restriction along a functor which preserves finite products, and furthermore conservative, we see that also $\widetilde{R} F$ commutes with finite coproducts. So $\widetilde{L} \widetilde{R} F$ lies in the essential image of $\widetilde{L}^c$. It suffices to check that the adjunction counit $\widetilde{L} \widetilde{R} F \rightarrow F$ is an equivalence. Since $p$ is conservative, we may check that $p\widetilde{L} \widetilde{R} F \simeq pF$ is an equivalence instead. Using the equivalences $p \widetilde{L} \widetilde{R} \simeq L q \widetilde{R} \simeq LRp$ we observe that this is equivalent to the adjunction counit inducing an equivalence $LRpF \rightarrow pF$. But if $H \colon \spc^2 \rightarrow \category{C}$ commutes with colimits, then $LRH \rightarrow H$ is an equivalence, which establishes the claim.
\end{proof}

Note that the functor $j\colon \Delta^1 \rightarrow \finpairs$ given by the morphism $(\emptyset \rightarrow \ast) \rightarrow (\ast \rightarrow \ast)$ exhibits the target as a free coproduct completion of the source.

\begin{appendixlemma}
	\label{appendixlem:finite_pairs_with_transfers}
	Let $\category{C}$ be a semiadditive category. Restriction along $ \Delta^1 \xrightarrow{j} \finpairs \rightarrow \finpairstrans$ induces an equivalence
	\[ \func^\cup(\finpairstrans,\category{C}) \simeq \func(\Delta^1,\category{C}). \]
\end{appendixlemma}

\begin{proof}
	Bachmann-Hoyois \cite[Prop. C.5]{BH21} construct an equivalence
	\[ \func^\cup( \finpairstrans, \category{C}) \simeq \func^{\cup}(\finpairs,\mathrm{CMon}(\category{C})). \]
	Since $\category{C}$ is semiadditive, $\mathrm{CMon}(\category{C}) \simeq \category{C}$. Combining this with the observation that $\finpairs$ is the free coproduct completion of $\Delta^1$ yields the claim.
\end{proof}

Let $\category{C}$ be a cocomplete semiadditive category. We write
$\func^\codesc(\spcpairtrans,\category{C}) \subset \func(\spcpairtrans,\category{C})$ for those functors whose restriction to $\spc^{2}$ commutes with colimits.

\begin{appendixproposition}
	\label{prop:constructing_functors_by_describing_map}
	
	Let $\category{C}$ be a cocomplete semiadditive category. 
	Restriction along the functors $\Delta^1 \subset \spc^2 \rightarrow \spcpairtrans$ induces equivalences of categories
	\begin{equation}
		\func^\codesc(\spcpairtrans,\category{C}) \simeq \func^L(\spc^2,\category{C}) \simeq \func(\Delta^1,\category{C}).
	\end{equation}
\end{appendixproposition}

\begin{proof}
	The first part is a straightforward combination of \cref{lem:construction_of_functors_with_transfers} and \cref{appendixlem:finite_pairs_with_transfers}, as well as the fact that $\Delta^1 \subset \spc^2$ exhibits the target as the colimit cocompletion of the source.
\end{proof}

\begin{appendixproposition}
	\label{prop:functor_out_of_pairs_with_transfers_satisfying_codescent}
	Let $\category{C}$ be a cocomplete semiadditive category. The inclusion
	\[ \func^{\codesc}(\spcpairtrans,\category{C}) \subset \func^{\cup}(\spcpairtrans,\category{C}) \]
	admits a right adjoint. Its restriction to $\spc^{2}$ is the left Kan extension of its restriction along the composite $\Delta^1 \subset \spc^2 \rightarrow \spcpairtrans$.
\end{appendixproposition}

\begin{proof}
	Our considerations identify the inclusion in question with the left adjoint
	\[ \func^\cup(\finpairstrans,\category{C}) \xrightarrow{\widetilde{L}} \func(\spcpairtrans,\category{C}). \qedhere \]
\end{proof}

As an application, we derive the classical pull-push formula for the transfer in singular homology. Letting $\category{C} = \module{\bbZ}$ and
 considering the functor $(\bbZ \rightarrow \bbZ) \colon \Delta^1 \rightarrow \module{\bbZ}$, \cref{prop:constructing_functors_by_describing_map} provides a functor $\bbZ \otimes (-) \colon \spcpairtrans \rightarrow \module{\bbZ}$. Note that, as plain $\bbZ$-modules we have $\bbZ \otimes (X,A) \simeq \bbZ \otimes \Sigma^\infty(X_+/A_+)$.  We will denote the forward functoriality by $(-)_!$ and the backward functoriality along finite covering maps by $(-)^{*}$.

\begin{appendixproposition}
	\label{appendixprop:covering_and_transfer_for_singular_homology}
	Let $p \colon (X',A') \rightarrow (X,A)$ be an $n$-sheeted covering of pairs. Then the composite
	\begin{equation}
		\bbZ \otimes (X,A) \xrightarrow{p^*} \bbZ \otimes (X',A') \xrightarrow{p_!} \bbZ \otimes (X,A)
	\end{equation}
	is multiplication by $n$.
\end{appendixproposition}

\begin{proof}
	Write $\spc^{2}_{/B\Sigma_n} = \func(\Delta^1, \spc_{/B\Sigma_n})$, i.e. the category of pairs of spaces $(X,A)$ together with an $n$-sheeted covering of pairs $(X',A') \rightarrow (X,A)$. We have the functor $T \in \func(\spc^2_{B\Sigma_n},\module{\bbZ}^{B\mathbb{N}})$ sending $((X',A') \xrightarrow{p} (X,A))$ to $\bbZ \otimes (X,A)$ equipped with the endomorphism $p_! p^*$. 
	The functor $T'$ sending $(X,A) \in \spc^2_{/B\Sigma_n}$ to $n \otimes (X,A) \colon \bbZ \otimes (X,A) \rightarrow \bbZ \otimes (X,A)$ preserves colimits as well, and we want to argue that $T\simeq T'$.
	
	Note that $\spc^2_{/B\Sigma_n}$ is freely generated under colimits by the subcategory $\Delta^1_{B\Sigma_n}$ on objects of the shape $d_\alpha = (\emptyset \rightarrow \ast \xrightarrow{\alpha} B\Sigma_n)$ and $e_\alpha = (\ast \rightarrow \ast \xrightarrow{\alpha} B\Sigma_n)$. The restrictions of $T$ and $T'$ to $\Delta^1_{B\Sigma}$ land in the $1$-category $\func(\Delta^1_{B\Sigma},\mathrm{Ab}) \simeq \func(\Delta^1, \mathrm{Ab}^{B\Sigma_n})$, where $\mathrm{Ab} = \module{\bbZ}^{\heartsuit}$ is the category of abelian groups.
	Indeed, as $\bbZ$-modules with $B\mathbb{N}$-action, $T(d_\alpha) = (\bbZ \xrightarrow{n} \bbZ)$ and $T(e_\alpha) = 0$, and $T'$ takes the same values.  Hence, the restrictions of $T$ and $T'$ are equivalent, and since both preserve colimits, we have $T\simeq T'$.
\end{proof}

\printbibliography
	
\end{document}